\documentclass[a4paper,11pt]{amsart}
\usepackage[utf8]{inputenc}
\usepackage{amsfonts}
\usepackage{hyperref}
\usepackage{amsmath}
\usepackage{amsthm}
\usepackage[capitalize,nameinlink]{cleveref}
\usepackage{comment}
\usepackage{mathrsfs}
\usepackage{mathabx}
\usepackage{multicol}
\usepackage{tikz}
\usepackage{leftindex}
\usepackage{enumitem}
\usepackage{todonotes}
\usetikzlibrary{positioning}
\usepackage{caption}
\newtheorem{theorem}{Theorem}[section]
\usepackage[a4paper, margin=1in]{geometry}
\newtheorem{corollary}[theorem]{Corollary}
\newtheorem{lemma}[theorem]{Lemma}
\newtheorem{proposition}[theorem]{Proposition}
\theoremstyle{definition}
\newtheorem{remark}[theorem]{Remark}

\newtheorem{definition}[theorem]{Definition}
\newtheorem{example}[theorem]{Example}

\renewcommand{\d}[1]{\ensuremath{\operatorname{d}\!{#1}}}

\newcommand{\ceil}[1]{\left\lceil {#1} \right\rceil}
\newcommand{\lin}[3]{\leftindex_{#1}{\left \langle #2,#3 \right\rangle}}
\newcommand{\rin}[3]{\left\langle #2,#3 \right\rangle_{#1}}
\newcommand{\op}[1]{\operatorname{#1}}
\newcommand{\Addresses}{
\setlength{\parindent}{0pt}{
\bigskip
    \footnotesize

    Arvin Lamando\par\nopagebreak

     \textsc{Acoustics Research Institute, Austrian Academy of Sciences, Dominikanerbastei 16,
    1010 Vienna, Austria;} \par\nopagebreak
     
    \textsc{Faculty of Mathematics, University of Vienna, Oskar-Morgenstern-Platz 1, 1090 Vienna, Austria;}\par\nopagebreak

    \textsc{Institute of Mathematics, University of the Philippines, Diliman, 1101 Quezon City, Philippines}\par\nopagebreak
    
    \textit{E-mail address}: 
    \texttt{arvin.lamando@oeaw.ac.at; arvin.lamando@univie.ac.at; alamando@math.upd.edu.ph}

}

{
    \bigskip
    \footnotesize

    Henry McNulty, \par\nopagebreak
    \textsc{Cognite AS, 1366 Lysaker, Norway}\par\nopagebreak
    \textsc{Department of Mathematical Sciences, Norwegian University of Science and Technology, 7491
    Trondheim, Norway}\par\nopagebreak
    \textit{E-mail address}: \texttt{henry.mcnulty@cognite.com}
}}

\title{On Modulation and Translation Invariant Operators and the Heisenberg Module}
\author{
  Arvin Lamando \and Henry McNulty 
  }
\thanks{This research was made possible through financial support from the OeAD-GmbH (Austria’s Agency for Education and Internationalisation) under the Ernst-Mach Grant ASEA-UNINET, which funded the PhD studies of the first named author.}
\begin{document}

\begin{abstract}
    We investigate spaces of operators which are invariant under translations or modulations by lattices in phase space. The natural connection to the Heisenberg module is considered, giving results on the characterisation of such operators as limits of finite--rank operators. Discrete representations of these operators in terms of elementary objects and the composition calculus are given. Different quantisation schemes are discussed with respect to the results.
\end{abstract}

\maketitle

\section{Introduction}
Quantum Harmonic Analysis (QHA), introduced by Werner in \cite{We84}, extends operations of classical harmonic analysis, the convolution and Fourier transform, to operators, such that the interaction between QHA and classical operations interact as one would expect. Central to QHA is the Weyl quantisation, associating to each function (called the Weyl symbol) a corresponding operator such that the translations of operators introduced in \cite{We84} correspond to a translation of the Weyl symbol, and the Fourier transform of an operator corresponds to the Fourier transform of its Weyl symbol. By considering the operations of QHA on rank-one operators, many fundamental objects of Time--Frequency Analysis (TFA) can be retrieved \cite{LuSk18, LuSk19}, which has led to efforts examining further the connections between QHA and TFA \cite{Sk20, Sk21}. Conversely, by instead working on the Weyl symbol level, operator compositions were examined in \cite{HoToWa07, To13}. In addition to the Weyl quantisation considered in this work, quantisations can be carried out on symbols from other classes of groups \cite{FiRu16, RuTuWi14, RuTu07}, and considering representations of other locally compact groups gives rise to different versions of QHA \cite{Ha23, BeBeLuSk22, FuGa23}. In \cite{LuMc24}, the concept of modulation of an operator was considered to extend QHA to Quantum Time--Frequency Analysis, corresponding to a modulation of the Weyl symbol of the operator. 

In \cite{FeKo98}, operators which were invariant under translations by some lattice $\Lambda$ were considered. The canonical $\Lambda$-translation invariant operator is the Gabor frame operator, defined in \cref{gab-anal-intro}, which is ubiquitous in the field of Gabor analysis, wherein one aims to discretise many of the operations from time--frequency analysis. As the Weyl symbols of such operators are $\Lambda$-periodic, the tools of Fourier analysis can be used to discretise the spreading function of such operators. In terms of QHA, this amounts to considering the Fourier series arising from the Fourier transform for operators \cite{Sk21}. Boundedness properties for $\Lambda$-translation invariant operators were studied in \cite{GaMo23}, while periodic elements of modulation spaces were studied in \cite{ToNa18}.

The Heisenberg modules, originally constructed by Rieffel in \cite{Ri88} as a tool to study the noncommutative tori are well-known to be connected to Gabor analysis \cite{Lu07, Lu09}. Using the general technique of localization of Hilbert $C^*$-modules, as was done in \cite{AuEn20,AuJaMaLu20}, any Heisenberg module can be continuously and densely embedded inside the space of square-integrable functions. Consequently, Gabor atoms coming from the Heisenberg modules generate continuous Gabor frame operators.

In this work, we introduce the connection between the Heisenberg modules and QHA. We start with the observation that all adjointable maps of the Heisenberg module associated to the lattice $\Lambda$ are exactly finite sums of Gabor frame operators on $\Lambda$ (Lemma \ref{lambda-trans-a-adjointable}). It then follows that adjointable maps of the Heisenberg module associated to the lattice $\Lambda$, are $\Lambda$-translation invariant operators. Conversely, we shall use techniques from operator algebras to show that we can characterise $\Lambda$-translation invariant operators as limits (under different operator topologies) of periodisations of finite rank operators, generated by functions from the Heisenberg module. Our operator algebraic techniques necessitate that we go beyond the usual Banach Gelfand triple framework for sequences, and in turn, we also obtain characterisation results beyond the usual setting of $\mathcal{M}^{\infty}$ for $\Lambda$-translation invariant operators.

We then introduce the notion of $\Lambda$-modulation invariant operators, using the concept of operator modulation analogous to the translation invariant setting. In examining such operators, the parity operator naturally arises as the only $\mathbb{R}^{2d}$-modulation invariant operator, as the identity is in the translation invariant case. This reflects the fundamental nature of the Weyl transform, and its relation to the spreading representation of an operator. 

We consider how $\Lambda$-modulation invariant operators can be represented and recovered in discrete settings, in particular, finding that they can be identified from discretisations of the convolution of QHA. We show that by choosing an appropriate lattice, $\Lambda$-modulation invariant operators have an interesting calculus, and show the explicit form of this. We shall also show that we can identify $\Lambda$-modulation invariant operators with $\Lambda/2$-translation invariant operators using the parity operator. The identification allows us to formulate (analogous) characterisations of $\Lambda$-modulation invariant operators in terms of a new kind of operator-periodisation based on operator-modulation, called the ``Fourier-Wigner operator-periodisation'' applied on finite-rank operators generated by the Heisenberg module. We shall also consider the discrepancy between the densities of the lattices $\Lambda/2$ and $\Lambda$, along with its implications on the minimum number of generators required to obtain the aforementioned characterisations for $\Lambda$-translation versus $\Lambda$-modulation invariant operators.

The relation between translation and modulation of operators depends fundamentally on the quantisation scheme used, and we discuss how Weyl quantisation is the natural setting for Quantum Time--Frequency Analysis.

\section{Preliminaries}
\subsection{On Duality}
In general, we will use the following notation given a Banach space $\mathcal{B}$:
\begin{align*}    \rin{\mathcal{B}',\mathcal{B}}{\varphi}{x} = \varphi(x) , \qquad \forall (\varphi,x)\in \mathcal{B}'\times \mathcal{B}
\end{align*}
where the brackets are antilinear in the second slot; this is tantamount to identifying the dual space $\mathcal{B}'$ with the continuous antilinear functionals. The rationale behind this identification is to preserve the inner product property in the case where $\mathcal{B}$ is a Hilbert space. 

In some cases, we will be working with Banach spaces $\mathcal{B}$ that embed to their dual $\mathcal{B}'$, hence there will be properties and constructions in $\mathcal{B}$ that we wish to naturally extend to the dual $\mathcal{B}'$. In particular, we extend bounded linear operators $T: \mathcal{B}_1\to \mathcal{B}_2$ to $T^{*}: \mathcal{B}_2'\to \mathcal{B}_1'$ using the usual Banach space adjoint operator determined by the following
\begin{align*}
\rin{\mathcal{B}'_1,\mathcal{B}_1}{T^*\varphi}{x} = \rin{\mathcal{B}_2',\mathcal{B}_2}{\varphi}{Tx}, \qquad \forall(\varphi,x) \in \mathcal{B}_2'\times \mathcal{B}_1.
\end{align*}
In this paper, we shall suppress notation, and denote the extension $T^*$ by $T$ still. Such extensions are commonly referred to as \emph{by-duality} extensions, since they are formulated using the duality bracket $\rin{\mathcal{B}',\mathcal{B}}{\cdot}{\cdot}$ between a space and its dual. Indeed, we are always implicitly extending operators dually when appropriate.

\subsection{Time--Frequency Analysis}
We begin by introducing the basic objects in time--frequency analysis. The translation and modulation operators are defined as 
\begin{align*}
    T_x f(t) :&= f(t-x) \\
    M_{\omega} \psi(t) :&= e^{2\pi i \omega\cdot t}f(t),
\end{align*}
for $x,\omega\in \mathbb{R}^d$ and $f\in L^2(\mathbb{R}^{d})$, and can be extended by duality to tempered distributions. The composition of the two gives a time--frequency shift:
\begin{align*}
    \pi(z) := M_\omega T_x
\end{align*}
where $z=(x,\omega)$, which are unitary on $L^2(\mathbb{R}^d)$ with adjoint $\pi(z)^*=e^{-2\pi i x\cdot \omega}\pi(-z)$. Furthermore, $\pi: \mathbb{R}^{2d}\to \mathcal{L}(L^2(\mathbb{R}^d))$ defines a strongly continuous map. The \emph{Short Time Fourier Transform} (STFT) of $f\in L^2(\mathbb{R}^d)$ with respect to the window $g\in L^2(\mathbb{R}^d)$ can then be defined as
\begin{align*}
    V_g f(z) = \langle f, \pi(z) g\rangle_{L^2}.
\end{align*}
The STFT induces a map \( f \mapsto V_g f \) from \( L^2(\mathbb{R}^d) \) into \( L^2(\mathbb{R}^{2d}) \); if \( \|g\|_{L^2} = 1 \), this map is an isometry. In addition, for any \( f, g \in L^2(\mathbb{R}^d) \), the function \( z \mapsto V_g f(z) \) is uniformly continuous on \( \mathbb{R}^{2d} \). Furthermore, the STFT satisfies Moyal's orthogonality relation
\begin{align*}
    \langle V_{g_1} f_1, V_{g_2} f_2 \rangle_{L^2} = \langle f_1, f_2\rangle\cdot\overline{\langle g_1, g_2\rangle_{L^2}},
\end{align*}
where the left hand side inner product is on $L^2(\mathbb{R}^{2d})$, while those on the right hand side are on $L^2(\mathbb{R}^{d})$. As a consequence, one finds the reconstruction formula
\begin{align}\label{contrecon}
    f = \frac{1}{\rin{L^2}{h}{g}} \int_{\mathbb{R}^{2d}} V_g f(z) \pi(z)h\, \d z,
\end{align}
for $\langle g, h\rangle \neq 0$. The integral in the right hand side of \cref{contrecon} can be interpreted weakly in the following sense: $\displaystyle h' \mapsto \frac{1}{\rin{L^2}{h}{g}}\int_{\mathbb{R}^{2d}} \rin{L^2}{V_gf(z) \pi(z)h}{h'}\d z$ defines a continuous linear functional in $L^2(\mathbb{R}^{2d})$, we define $\displaystyle \frac{1}{\rin{L^2}{h}{g}}\int_{\mathbb{R}^{2d}}V_gf(z)\pi(z)h \d z$ to be the corresponding element in $L^2(\mathbb{R}^d)$ by Riesz representation theorem.

Closely related to the STFT are the \emph{ambiguity function}
\begin{align*}
    A(f,g)(x,\omega) = e^{i\pi x\cdot \omega}V_g f(x,\omega),
\end{align*}
and the \emph{cross-Wigner distribution}
\begin{align*}
    W(f,g)(x,\omega) :&= \int_{\mathbb{R}^{2d}} f(x+\tfrac{t}{2})\overline{g(x-\tfrac{t}{2})}e^{-2\pi i \omega t}\d z.
\end{align*}

The \emph{Weyl symbol} $\sigma_S$ of an operator $S\in\mathcal{L}(\mathscr{S}(\mathbb{R}^d),\mathscr{S}^\prime(\mathbb{R}^d))$, where $\mathscr{S}$ is the Schwartz space, can be defined weakly as the unique distribution such that
\begin{align*}
    \langle Sf,g\rangle_{\mathscr{S}^\prime,\mathscr{S}} = \langle \sigma_S, W(g,f)\rangle_{\mathscr{S}^\prime,\mathscr{S}}.
\end{align*}
Note that the mapping $S\mapsto \sigma_S$ is a unitary map from the space of Hilbert-Schmidt operators $\mathcal{HS}$ to $L^2(\mathbb{R}^{2d})$.
By taking a Schwartz function window  $\varphi_0\in\mathscr{S}(\mathbb{R}^{d})$, the STFT $V_{\varphi_0}f$ can be extended to the space of tempered distributions $f\in \mathscr{S}'(\mathbb{R}^d)$. In particular, one can take $\varphi_0\in \mathscr{S}(\mathbb{R}^d)$ to be an $L^2(\mathbb{R}^d)$ normalised Gaussian to define the \emph{modulation spaces}
\begin{align*}
M^{p,q}_m(\mathbb{R}^{d}) = \{ f\in \mathscr{S}^\prime(\mathbb{R}^{d}): \|f\|_{M^{p,q}_m} := \|V_{\varphi_0} f\|_{L^{p,q}_m(\mathbb{R}^{2d})} < \infty \},
\end{align*}
where $L^{p,q}_m(\mathbb{R}^{2d})$ is the weighted mixed--norm Lebesgue space (cf. Chapter 11 \cite{Gr01}). All Schwartz functions can be used as windows to define the same spaces with equivalent norms, and the dual of $M^{p,q}_m(\mathbb{R}^d)$ is $M^{p',q'}_{1/m}(\mathbb{R}^{d})$ for $1\leq p,q < \infty$ where $p',q'$ are the appropriate H\"older conjugates. Furthermore, for $1\leq p_1 \leq p_2 \leq \infty$ and $1\leq q_1 \leq q_2 \leq \infty$, we have the continuous inclusions \cite{Gr01}: $M^{p_1,q_1}(\mathbb{R}^d) \hookrightarrow M^{p_2,q_2}(\mathbb{R}^d)$. Of particular interest is the so-called \emph{Feichtinger's algebra} $M^1(\mathbb{R}^{d}) := M^{1,1}_{1}(\mathbb{R}^{2d})$ with dual $M^{\infty}(\mathbb{R}^d)$, which generates the \emph{Banach Gelfand triple} $(M^1,L^2, M^\infty)(\mathbb{R}^{d})$ \cite{FeLuCo08}, where $M^\infty(\mathbb{R}^{d})$ is equipped with the weak-$*$ topology. With slight modifications, for a fixed $1\leq p \leq \infty$, we can also define modulation spaces $M^p(G)$ for a locally compact abelian group $G$ (see \cite{Fe03}, and \cite{FeJa22} for a recent treatment).

\subsection{Harmonic Analysis on Phase Space}
This work is primarily concerned with phase space $\mathbb{R}^d \times \widehat{\mathbb{R}}^d$, where $\widehat{\mathbb{R}}^d$ denotes the Pontryagin dual of $\mathbb{R}^{d}$. We will usually abbreviate and identify the phase space $\mathbb{R}^{d}\times \widehat{\mathbb{R}}^d$ with $\mathbb{R}^{2d}$, given the fact that $\mathbb{R}^{d}\cong \widehat{\mathbb{R}}^d$. The \emph{symplectic form} on the phase space is the mapping $\Omega: \mathbb{R}^{2d}\times \mathbb{R}^{2d}\to \mathbb{R}$, defined via 
\begin{align*}
    \Omega(z,z') := x'\cdot \omega - x\cdot \omega',
\end{align*}
for $z=(x,\omega), z'=(x',\omega')\in\mathbb{R}^{2d}$. The symplectic form $\Omega$ plays a 
central role in the harmonic analysis on $\mathbb{R}^{2d}$ since it appears in the the appropriate notion of a Fourier transform on the phase space. The \emph{symplectic Fourier transform} is defined by
\begin{align*}
    \mathcal{F}_{\Omega} f(z) := \int_{\mathbb{R}^{2d}} f(z')\, e^{-2\pi i \Omega(z,z')}\, \d z',
\end{align*}
and enjoys many properties analogous to those of the standard Euclidean Fourier transform. However, unlike the Euclidean Fourier transform, the symplectic Fourier transform is its own inverse. Note now that the ambiguity function and the cross-Wigner distribution are connected via the symplectic Fourier transform:
\begin{align*}
     W(f,g)(x,\omega)= \mathcal{F}_{\Omega}\big(A(f,g)\big)(x,\omega).
\end{align*}
A \emph{lattice} in $\mathbb{R}^{2d}$ will always refer to a full--rank lattice $\Lambda = A\mathbb{Z}^{2d}$ for some $A\in \op{GL}(2d,\mathbb{R})$, and the lattice volume is given by $|\Lambda|:=|\det(A)|$. The \emph{adjoint lattice} $\Lambda^{\circ}$ is isomorphic to the annihilator subgroup $\Lambda^{\perp}$ of $\Lambda$, defined via 
\begin{align*}
    \Lambda^\circ :=& \{z\in\mathbb{R}^{2d}: \pi(z)\pi(\lambda)=\pi(\lambda)\pi(z),\, \forall \lambda\in\Lambda\} \\
    =& \{z\in\mathbb{R}^{2d}: e^{2\pi i \Omega(z,\lambda)}=1, \, \forall \lambda\in\Lambda\}.
\end{align*}
$\Lambda^\circ$ is then itself a lattice in $\mathbb{R}^{2d}$, with volume $|\Lambda^{\circ}|=\tfrac{1}{|\Lambda|}$. Importantly, $\Lambda^\circ$ can be identified with the dual group of $\mathbb{R}^{2d}/\Lambda$. We define the space $\mathcal{A}(\mathbb{R}^{2d}/\Lambda)$ as the space of $\Lambda$-periodic functions $f$ on $\mathbb{R}^{2d}$ of the form
\begin{align}\label{classical-sfs}
    f(z) = \sum_{\lambda^\circ\in\Lambda^\circ} c_{\lambda^\circ} e^{-2\pi i \Omega(z,\lambda^\circ)},
\end{align}
where $\{c_{\lambda^\circ}\}_{\lambda^\circ\in\Lambda^\circ}$, called the symplectic \emph{Fourier coefficients} of $f$, are absolutely convergent, with norm $\|f\|_{\mathcal{A}}=\|\{c_{\lambda^{\circ}}\}_{\lambda^{\circ}\in \Lambda^{\circ}}\|_{\ell^1(\Lambda^{\circ})}$. The dual of $\mathcal{A}(\mathbb{R}^{2d}/\Lambda)$ is denoted by $\mathcal{A}^\prime(\mathbb{R}^{2d}/\Lambda)$, which are functions of the form \cref{classical-sfs} with symplectic Fourier coefficients in $\ell^\infty(\Lambda^\circ)$. Given an $M^1(\mathbb{R}^{2d})$ function $f$, the periodisation 
\begin{align*}
    P_{\Lambda} f := \sum_{\lambda\in\Lambda} T_\lambda f
\end{align*}
is an element of $\mathcal{A}(\mathbb{R}^{2d}/\Lambda)$, and moreover the map $P_{\Lambda}: M^1(\mathbb{R}^{2d})\to \mathcal{A}(\mathbb{R}^{2d}/\Lambda)$ is bounded and surjective \cite{Fe81}. As a result of this, we can consider the symplectic Poisson summation formula for $f\in M^1(\mathbb{R}^{2d})$:
\begin{align}\label{symplpoisson}
    (P_{\Lambda} f)(z) = \frac{1}{|\Lambda|} \sum_{\lambda^\circ\in\Lambda^\circ} \mathcal{F}_{\Omega} (f)(\lambda^\circ)e^{2\pi i \Omega(z,\lambda^\circ)}, \qquad \forall z\in \mathbb{R}^{2d},
\end{align}
where the sum is absolutely convergent. The norm on $\Lambda^{\circ}$ is given by the dual measure of $\mathbb{R}^{2d}/\Lambda$, and so the $\ell^p(\Lambda^{\circ})$-norm is given by
\begin{align}\label{lp-for-adjoint-lattice}
\|\mathbf{k}\|_{\ell^p(\Lambda^{\circ})}:= \frac{1}{|\Lambda|^{1/p}}\left(\sum_{\lambda^{\circ}\in \Lambda^{\circ}}|k_{\lambda^{\circ}}|^p\right)^{1/p}
\end{align}
for a sequence $\mathbf{k} = \{k_{\lambda^{\circ}}\}_{\lambda^{\circ}\in \Lambda^{\circ}}$ for $1\leq p < \infty$ with the usual modification when $p=\infty.$

\subsection{Gabor Analysis}\label{gab-anal-intro} In time-frequency analysis, one is often concerned with analysis and synthesis on a lattice $\Lambda$ of phase space (or the time-frequency plane) via the STFT and the time-frequency shifts. We introduce to the \emph{Gabor frame operator} $S_{g,h,\Lambda}$ for $g,h\in L^2(\mathbb{R}^d)$;
\begin{align}\label{mixed-gabor}
    S_{g,h,\Lambda}f = \sum_{\lambda\in \Lambda}V_gf(\lambda)\pi(\lambda)h.
\end{align}
Comparing \cref{mixed-gabor} with \cref{contrecon}, the Gabor frame operator can be seen as a discretisation of the continuous reconstruction formula, $S_{g,h}f$ analyses $f$ by the samples of its STFT $\{V_gf(\lambda)\}_{\lambda\in \Lambda}$ with respect to the analysing window $g$, and then synthesises using the sampled time--frequency shifts of the synthesising window $h.$ A classical problem of Gabor analysis is to find a single \emph{Gabor atom} $g \in L^2(\mathbb{R}^d)$, in a given lattice $\Lambda$, such that $S_{g,g,\Lambda} =: S_{g,\Lambda}$ is an invertible bounded linear operator in $L^2(\mathbb{R}^d).$ This is equivalent to finding an $h\in L^2(\mathbb{R}^d)$, called a \textit{dual Gabor atom} (with respect to $g$), where $S_{g,h,\Lambda}=\op{Id}_{L^2(\mathbb{R}^d)}.$ Finding such a $g$ is tantamount to solving the so-called \emph{Gabor frame inequalities}, and would give us perfect reconstruction of any $f\in L^2(\mathbb{R}^d)$ via:
\begin{align*}
    f &= S_{g,\Lambda}^{-1}S_{g,\Lambda}f = \sum_{\lambda\in \Lambda}V_gf(\lambda)\pi(\lambda)S^{-1}_{g,\Lambda}g\\
    &=S_{g,\Lambda}S_{g,\Lambda}^{-1}f = \sum_{\lambda\in \Lambda}V_{S_{g,\Lambda}^{-1}g}(f)\pi(\lambda)g.
\end{align*}
It is well-known that the Feichtinger's algebra $M^1(\mathbb{R}^d)\subseteq L^2(\mathbb{R}^d)$ is a particularly useful class of function space for Gabor atoms since they always give us continuous Gabor frame operators \cite[Theorem 3.3.1]{FeZi98}. For our later reference, we include this as a proposition:
\begin{proposition}\label{feich-nice}
    For \emph{all} lattices $\Lambda\subseteq \mathbb{R}^{2d}$ and any $g,h\in M^1(\mathbb{R}^d)$, we have that $S_{g,h,\Lambda}\in \mathcal{L}\left(L^2(\mathbb{R}^d)\right).$
\end{proposition}

\subsection{Heisenberg Modules} The Heisenberg modules were originally constructed by Rieffel in \cite{Ri88} to study the finitely generated projective modules of the higher dimensional noncommutative tori. Their connections to time--frequency analysis was first highlighted by Luef in \cite{Lu07}. We shall use the concrete description of the Heisenberg modules using the results of Austad and Enstad in \cite{AuEn20}, obtained through \emph{localisation}. 

Let us briefly discuss Hilbert $C^*$-modules as they provide the general abstract framework for the Heisenberg modules. See \cite{RaWi98,La95} for more details, and for examples.
\begin{definition}\label{operator-inner-product}
    Fix a $C^*$-algebra $A$, a complex vector space $X$ is called an \emph{inner product left} $A$\emph{-module} if it is a left $A$-module $X$ compatible with scalar multiplication equipped with a map $\lin{A}{\cdot}{\cdot}:A\times A\to X$ satisfying the following axioms:
\begin{enumerate}[label=(\alph*)]
			\item $\lin{A}{\lambda z+ \mu y}{w}=\lambda \lin{A}{z}{w} + \mu\lin{A}{y}{w} $; 
			\item $\lin{A}{a z}{w}= a \lin{A}{z}{w}$;
			\item $\lin{A}{z}{w}^* = \lin{A}{w}{z}$;
			\item $\lin{A}{z}{z}\geq 0$; 
                \item $\lin{A}{z}{z}=0 \implies z=0$,
		\end{enumerate}
        for all $a\in A$, $z,w,y\in X$, and $\lambda,\mu\in \mathbb{C}.$
\end{definition}
\noindent There is a natural way of inducing a norm on $X$ using the norm $\|\cdot \|$ on $A$ given by
\begin{align}\label{left-module-norm}
\|z\|_X := \|\lin{A}{z}{z}\|^{1/2}.
\end{align}
If $X$ is complete with respect to the $\|\cdot\|_X$-norm, then we call $X$ a \emph{Hilbert left} $C^*$\emph{-module over} $A$, or simply \emph{Hilbert left} $A$\emph{-module}. Notice that for $X$, the $C^*$-algebra valued inner product $\lin{A}{\cdot}{\cdot}$ is linear in the first slot and antilinear in the second slot. Suppose that $B$ is another $C^*$-algebra, there is an analogous notion of \emph{inner product right} $B$\emph{-module}. Say $Y$ is such a module, then it is a right $B$-module which is also a complex vector space with a $B$-valued inner product $\rin{B}{\cdot}{\cdot}:Y\times Y\to B$ satisfying similar axioms to Definition \ref{operator-inner-product}, where the key difference is that $\rin{B}{\cdot}{\cdot}$ is linear in the second slot while antilinear in the first. If $Y$ is complete with respect to the analogous norm induced by $\rin{B}{\cdot}{\cdot}$, then $Y$ is called a \emph{Hilbert right} $C^*$\emph{-module over} $B$ or \emph{Hilbert right} $B$\emph{-module}. It is now apparent that Hilbert $C^*$-modules generalize Hilbert spaces, since Hilbert $\mathbb{C}$-modules are precisely Hilbert spaces.
\begin{definition}
    Let $X$ and $Y$ be Hilbert left $A$-modules, the map $T: X\to Y$ is called an \emph{adjointable} or $A$\emph{-adjointable} map from $X$ to $Y$ if there exists another map called the \emph{adjoint} $T^*: Y\to X$ such that for all $z\in X$ and $w\in Y$:
    \begin{align*}
        \lin{A}{Tz}{w} = \lin{A}{z}{T^*w}.
    \end{align*}
    We denote the set of all $A$-adjointable maps via $\mathcal{L}_A(X,Y)$, with $\mathcal{L}_A(X):= \mathcal{L}_A(X,X)$. There is an analogous definition for Hilbert right $C^*$-modules.
\end{definition}
One sees that adjointable maps are the proper morphisms between Hilbert $C^*$-modules since it can be shown that they admit unique adjoint, are automatically continuous, and are linear both with respect to the module action and the scalar multiplication. Note that unlike the Hilbert space case, bounded $A$-linear functions between Hilbert $C^*$-modules do not automatically admit an adjoint in the sense given above \cite{La95}. 

\begin{definition}\label{fullness}
    Let $A$  be a $C^*$-algebra and $X$ a Hilbert left $A$-module, then $X$ is said to be \emph{full} if the ideal $\lin{A}{Z}{Z}=\op{span}\{\lin{A}{z}{w}:z,w\in Z\}$ is dense in $A$. We have an analogous definition for Hilbert right $C^*$-modules.
\end{definition}
\begin{definition}\label{imprimitivity-bimodule}
        Let $A,B$ be $C^*$-algebras, if $X$ is a complex vector space satisfying the following:
        \begin{enumerate}[label=(\alph*)]
            \item $X$ is a full Hilbert left $A$-module and a full Hilbert right $B$-module;
            \item For all $a\in A$, $b\in B$, $z,w\in X$:
            \begin{align*}
                \rin{B}{a z}{w} = \rin{B}{z}{a^*w},\ \text{and } \lin{A}{zb}{w} = \lin{A}{z}{wb^*};
            \end{align*}
            \item For all $z,y,w\in X$:
            \begin{align}\label{for-janssen}
            \lin{A}{z}{y}w = z \rin{B}{y}{w},
            \end{align}
        \end{enumerate}
        then $X$ is called an $A-B$\emph{-equivalence bimodule}.
\end{definition}
As is common in analysis, such structures can be defined through dense subspaces. Suppose $A_0$ and $B_0$ are dense $*$-subalgebras of $A$ and $B$, respectively, the structure that we have in mind is called an $A_0-B_0$\emph{-pre-equivalence bimodule}. It is an $A_0-B_0$-bimodule with $A_0$ and $B_0$ semi-inner products satisfying several symmetry and continuity axioms; see \cite[Definition 2.2]{AuJaMaLu20} or \cite[Definition 3.9]{RaWi98}) for a definition. If $X_0$ is an $A_0-B_0$-pre-equivalence bimodule, then it can be completed to an $A-B$-equivalence bimodule $X$ that preserves the structure of $X_0$ \cite[Proposition 3.12]{RaWi98}.

Note that if $X$ is an $A-B$-equivalence bimodule, there is no ambiguity in its norm as a Hilbert $C^*$-module since the induced norm from the inner product left $A$-module structure and the inner product right $B$-module structure coincide (see \cite[Proposition 3.11]{RaWi98}). 

It is well known that if at least one of the $C^*$-algebras $A$ or $B$ is unital, then any $A-B$-equivalence bimodule $X$ is finitely generated and projective (as a module with respect to the other $C^*$-algebra) \cite{Ri82,AuJaMaLu20}. Due to Serre-Swan's theorem \cite{Se55, Sw62}, finitely generated projective modules over commutative unital $C^*$-algebras correspond to complex vector bundles over the spectrum of such algebras. One can then see that $X$ can be treated as a complex vector bundle over a generally noncommutative $C^*$-algebra, hence the importance of such objects for noncommutative geometers. Furthermore, for operator algebraists, if an $A-B$-equivalence bimodule exists, then $A$ and $B$ are said to be \emph{Morita equivalent} \cite{Rie74,RaWi98}. The symmetrical structure of such a bimodule induces a correspondence on the ideal and representation structures between the two $C^*$-algebras. In this paper, we are interested in a particular equivalence bimodule, called the \emph{Heisenberg module}, primarily because its structure is closely related to Gabor analysis.

To start, we need to define our relevant $C^*$-algebras $A$ and $B$ for a fixed lattice $\Lambda\subseteq \mathbb{R}^{d}\times \widehat{\mathbb{R}}^{d}$. We introduce the \emph{Heisenberg} $2$\emph{-cocycle} $c: \mathbb{R}^{2d}\times \mathbb{R}^{2d}\to \mathbb{T}$
\begin{align}\label{heis-2-cos}
    c(z_1,z_2) = e^{-2\pi i x_1\cdot \omega_2}
\end{align}
for $z_1 = (x_1,\omega_1), z_2 = (x_2,\omega_2)\in \mathbb{R}^{2d}.$ Fixing $z_3=(x_3,\omega_3)\in \mathbb{R}^{2d}$, the Heisenberg $2$-cocycle satisfies the usual normalized $2$-cocycle conditions:
\begin{align*}
c(z_1,z_2)c(z_1+z_2,z_3)&=c(z_1,z_2+z_3)c(z_2,z_3)\\
c(z_1,0)=c(0,0)&=c(0,z_2)=1.
\end{align*}
It is related to the time-frequency shifts via
\begin{align}
\pi(z_1)\pi(z_2)&=c(z_1,z_2)\pi(z_1+z_2)\label{c-proj} \\
&= c(z_1,z_2)\overline{c(z_2,z_1)}\pi(z_2)\pi(z_1)\nonumber.
\end{align}
We see that $\pi: \mathbb{R}^{2d}\to \mathcal{L}\left(L^2(\mathbb{R}^d)\right)$ is not a unitary representation of the phase space $\mathbb{R}^{2d}=\mathbb{R}^d\times \widehat{\mathbb{R}}^d$, but instead \cref{c-proj} makes $\pi$ a $\mathit{c}$\emph{-projective representation} of $\mathbb{R}^{2d}$ on $L^2(\mathbb{R}^d)$. Such representations can be seen as a particular instance of so-called \emph{covariant representations} of twisted $C^*$-dynamical systems \cite{PaRa89}. Let us now fix a lattice $\Lambda$ of $\mathbb{R}^{2d}$, and consider the sequence space $\ell^1(\Lambda)$ equipped with the following $c$-twisted convolution and involution respectively:
\begin{align*}
    (\mathbf{a}_1 \natural_c \mathbf{a}_2)(\lambda) &= \sum_{\mu\in \mu}a_1(\mu)a_2(\lambda-\mu)c(\mu,\lambda-\mu)\\
    (\mathbf{a}_1)^{*_c}(\lambda) &= c(\lambda,\lambda)\overline{a_1(-\lambda)}
\end{align*}
for $\mathbf{a}_1,\mathbf{a}_2\in \ell^1(\Lambda).$ It can be shown that  $\ell^1(\Lambda)$ with the said structure along with its original norm gives us a Banach $*$-algebra, which we denote by $\ell^1(\Lambda,c)$. We similarly obtain another Banach $*$-algebra $\ell^1(\Lambda^{\circ},\overline{c})$ constructed by equipping $\ell^1(\Lambda^{\circ})$ with the following $\overline{c}$-twisted structure:
\begin{align*}
    (\mathbf{b}_1 \natural_{\overline{c}} \mathbf{b}_2)(\lambda^{\circ}) &= \frac{1}{|\Lambda|}\sum_{\mu^{\circ}\in \mu^{\circ}}b_1(\mu^{\circ})b_2(\lambda^{\circ}-\mu^{\circ})\overline{c(\mu^{\circ},\lambda-\mu^{\circ})}\\
    (\mathbf{b}_1)^{*_{\overline{c}}} (\lambda^{\circ}) &=\overline{c(\lambda^{\circ},\lambda^{\circ})}\overline{b_1(-\lambda^{\circ})}
\end{align*}
for $\mathbf{b}_1,\mathbf{b}_2\in \ell^1(\Lambda^{\circ}).$
\begin{definition}
    The $C^*$\emph{-enveloping algebra} (or the $C^*$\emph{-completion} \cite[Definition 10.4]{1fell88}) of $\ell^1(\Lambda,c)$ will hereafter be denoted by $A.$ Similarly, the $C^*$-enveloping algebra of $\ell^1(\Lambda^{\circ},\overline{c})$ will hereafter be denoted by $B.$
\end{definition}
\begin{remark}
    We note that both $\ell^1(\Lambda,c)$ and $\ell^1(\Lambda^{\circ},\overline{c})$ are unital Banach $*$-algebras, therefore, their $C^*$-envelopes, $A$ and $B$ respectively, must both be unital as well.
\end{remark}

The integrated forms of the time-frequency shifts, given by $\overline{\pi}_{A}:A\to \mathcal{L}\left(L^2(\mathbb{R}^d)\right)$ and $\overline{\pi}^*_B:B\to \mathcal{L}\left(L^2(\mathbb{R}^d)\right)$, are densely defined via
\begin{align*}
\overline{\pi}_A(\mathbf{a})&=\sum_{\lambda\in \Lambda}a(\lambda)\pi(\lambda) \\
\overline{\pi}^*_{B}(\mathbf{b})&=\frac{1}{|\Lambda|}\sum_{\lambda^{\circ}\in \Lambda^{\circ}} b(\lambda^{\circ})\pi^*(\lambda^{\circ})
\end{align*}
for $\mathbf{a}\in \ell^1(\Lambda,c)$, and $\ell^1(\Lambda^{\circ},\overline{c})\in B$. It is known that $\overline{\pi}_A$ is a $*$-representation, while $\overline{\pi}^*_B$ is a $*$-antirepresentation (meaning it reverses the multiplication on $B$). Furthermore, they are also well known to be faithful \cite{Ri88, GrLe04}, and hence are isometries.

The Heisenberg module will be densely defined through the  Feichtinger's algebra $M^1(\mathbb{R}^d)$. Consider the maps $\lin{\Lambda}{\cdot}{\cdot}: M^1(\mathbb{R}^d)\times M^1(\mathbb{R}^d)\to \ell^1(\Lambda,c)$ and $\rin{\Lambda^{\circ}}{\cdot}{\cdot} :M^1(\mathbb{R}^d)\times M^1(\mathbb{R}^d)\to \ell^1(\Lambda^{\circ},\overline{c})$ and put:
\begin{multicols}{2}
\begin{enumerate}
[label=(\alph*)]
\item $\lin{\Lambda}{f}{g}(\lambda) := \rin{L^2}{f}{\pi(\lambda)g}$ 
\item $\rin{\Lambda^{\circ}}{f}{g} := \rin{L^2}{g}{\pi^*(\lambda^{\circ})f}$
\end{enumerate}
\end{multicols}
\noindent for all $f,g\in M^1(\mathbb{R}^d)$ and $(\lambda,\lambda^{\circ})\in \Lambda\times \Lambda^{\circ}$. The Banach $*$-algebras $\ell^1(\Lambda,c)$ and $\ell^1(\Lambda^{\circ},\overline{c})$ also act on $M^1(\mathbb{R}^d)$ making it a left $\ell^1(\Lambda,c)$ and right $\ell^1(\Lambda^{\circ},\overline{c})$ module respectively. That is,
\begin{multicols}{2}
\begin{enumerate}
[label=(\alph*)]
\item $\displaystyle \mathbf{a}f := \overline{\pi}_A(\mathbf{a})f = \sum_{\lambda\in \Lambda}\pi(\lambda)f\in M^1(\mathbb{R}^d)$ \\
\item $\displaystyle f \mathbf{b} := \sum_{\lambda\in \Lambda^{\circ}}b(\lambda^{\circ})\pi^*(\lambda)f\in M^1(\mathbb{R}^d)$
\end{enumerate}
\end{multicols}
\noindent for all $f\in M^1(\mathbb{R}^d)$,  $\mathbf{a}\in \ell^1(\Lambda,c)$, and $\mathbf{b}\in \ell^1(\Lambda^{\circ},\overline{c}).$
The inner products and the module actions given above are well defined, following from the properties of Feichtinger's algebra (e.g. see \cite{Fe81} or \cite{Ja18}). Furthermore, it follows from the \emph{Fundamental Identity of Gabor Analysis} \cite[Theorem 4.5]{FeLu06} that for all $f,g,h\in M^1(\mathbb{R}^d)$:
\begin{align}\label{feich-figa}
    \lin{\Lambda}{f}{g}h = f\rin{\Lambda^{\circ}}{g}{h}.
\end{align}
It can be shown \cite{Lu07,JaLu21} that the structure that we have just defined above for $M^1(\mathbb{R}^d)$ is an $\ell^1(\Lambda,c)-\ell^1(\Lambda^{\circ},\overline{c})$-pre-equivalence bimodule.
\begin{definition}\label{abs-heis}
    For a fixed lattice $\Lambda\subseteq \mathbb{R}^d\times \widehat{\mathbb{R}}^d$, the Heisenberg module is an $A-B$-equivalence bimodule obtained by the completion of $M^1(\mathbb{R}^d)$ as an $\ell^1(\Lambda,c)-\ell^1(\Lambda,\overline{c})$-pre-equivalence bimodule. 
\end{definition}
\begin{remark}
    By definition, the defined inner products extend densely to $\lin{\Lambda}{\cdot}{\cdot}: \mathcal{E}_{\Lambda}(\mathbb{R}^d)\times \mathcal{E}_{\Lambda}(\mathbb{R}^d)\to A$ and $\rin{\Lambda^{\circ}}{\cdot}{\cdot}:\mathcal{E}_{\Lambda}(\mathbb{R}^d)\times \mathcal{E}_{\Lambda}(\mathbb{R}^d) \to B.$ The same can be said to the module actions.
\end{remark}
The Heisenberg module, as abstractly defined in \ref{abs-heis}, does not give us a concrete interpretation of the inner product and module structures when working with bonafide elements of $A,B$ and $\mathcal{E}_{\Lambda}(\mathbb{R}^d)$. A more grounded description of the Heisenberg module as a function space can be obtained through \cite[Propositions 3.3 and 3.7]{AuEn20}, showing that everything works as expected in the completion.
\begin{theorem}\label{heisenberg-module}
     The $C^*$-algebras $A$ and $B$ satisfy the following dense continuous embeddings:

        \begin{equation}\label{sequence-embeddings}
             \begin{split}
             \ell^1(\Lambda,c)\hookrightarrow A \hookrightarrow \ell^2(\Lambda),\qquad
             \ell^1(\Lambda^{\circ},\overline{c}) \hookrightarrow B\hookrightarrow \ell^2(\Lambda^{\circ}).
            \end{split}
         \end{equation}
     The Heisenberg module $\mathcal{E}_{\Lambda}(\mathbb{R}^d)$ can be obtained by completing the Feichtinger's algebra $M^1(\mathbb{R}^d)$ with respect to the norm
     \begin{align}\label{completion-for-heis}
    \|g\|_{\mathcal{E}_{\Lambda}(\mathbb{R}^d)}:= \|S_{g,\Lambda}\|_{\mathcal{L}(L^2)}^{1/2}. 
     \end{align}
     The Heisenberg module is densely and continuously embedded in $L^2(\mathbb{R}^d)$, with the following estimate:
     \begin{align*}
         \|f\|_{L^2} \leq \sqrt{|\Lambda|}\|f\|_{\mathcal{E}_{\Lambda}(\mathbb{R}^d)}, \qquad \forall f\in \mathcal{E}_{\Lambda}(\mathbb{R}^d).
     \end{align*}
 Furthermore, $\mathcal{E}_{\Lambda}(\mathbb{R}^d)$ is an $A-B$-equivalence bimodule with the following Hilbert left $A$-module and Hilbert right $B$-module structure: for $\lambda\in \Lambda$, $\lambda^{\circ}\in \Lambda^{\circ}$, $\mathbf{a}\in A$, $\mathbf{b}\in B$, and $f,g\in \mathcal{E}_{\Lambda}(\mathbb{R}^d)$:
     \begin{multicols}{2}
         \begin{enumerate}[label=(\alph*)]
             \item $\lin{\Lambda}{f}{g}(\lambda):=\rin{L^2}{f}{\pi(\lambda)g}$
            \item $\rin{\Lambda^{\circ}}{f}{g}(\lambda^{\circ}):= \rin{L^2}{g}{\pi^*(\lambda^{\circ})f}$         
            \item $\mathbf{a} f := \overline{\pi}_A(\mathbf{a})f$
        \item $f \mathbf{b} := \overline{\pi}^*_{B}(\mathbf{b})f.$
         \end{enumerate}
    \end{multicols}
 \end{theorem}
\noindent The norm defined in \cref{completion-for-heis} is well-defined due to \ref{feich-nice}. 

\begin{remark}
    It is important to note that in the general case, for any lattice $\Gamma \subseteq \mathbb{R}^{2d}$, and any $2$-cocyle $\theta: \Gamma\times \Gamma\to \mathbb{T}$, the twisted Banach $*$-algebra $\ell^1(\Gamma,\theta)$ is a $*$-semisimple $*$-algebra (see for example \cite{GrLe16} and \cite{SaWi20}) so that $\ell^1(\Gamma,\theta)$ is densely embedded in its $C^*$-envelope. The aforementioned $C^*$-enveloping algebra is isomorphic to a noncommutative $2d$-torus, hence our defined $C^*$-algebras $A$ and $B$ are themselves noncommutative $2d$-tori as well. The Heisenberg module as constructed by \cite{Ri88} was originally used to study noncommutative versions of complex vector bundles over noncommutative tori.
\end{remark} 

The next result from \cite[Theorem 3.15]{AuEn20} motivates the study of Heisenberg modules as a viable space of functions for Gabor analysis since it turns out that $\mathcal{E}_{\Lambda}(\mathbb{R}^d)$ inherits the crucial property \ref{feich-nice} of $M^1(\mathbb{R}^d)$. 
\begin{proposition}\label{gab-mix-adjointable}
    If $g,h\in \mathcal{E}_{\Lambda}(\mathbb{R}^d)$, then $S_{g,h,\Lambda}\in \mathcal{L}\left(L^2(\mathbb{R}^d)\right)$. We have that, 
    \begin{align*}
        \|g\|_{\mathcal{E}_{\Lambda}(\mathbb{R}^d)}=\|S_{g,g,\Lambda}\|^{1/2}.
    \end{align*}
    Furthermore, the restriction $(S_{g,h,\Lambda})_{|\mathcal{E}_{\Lambda}(\mathbb{R}^d)}$ is an $A$-adjointable map in $\mathcal{E}_{\Lambda}(\mathbb{R}^d).$
\end{proposition}
\begin{remark}
    The associative formula \cref{feich-figa} extends to $f,g,h\in \mathcal{E}_{\Lambda}(\mathbb{R}^d)$. When written explicitly, we have
    \begin{align*}
        \sum_{\lambda\in \Lambda}\rin{L^2}{f}{\pi(\lambda)g}\pi(\lambda)h = \frac{1}{|\Lambda|}\sum_{\lambda^{\circ}\in \Lambda^{\circ}}\rin{L^2}{h}{\pi(\lambda^{\circ})g}\pi(\lambda^{\circ})f.
    \end{align*}
    From which we can deduce that
    \begin{align}\label{janssen-for-heis}
        S_{g,h,\Lambda} = \overline{\pi}^{*}_{B}(\rin{B}{g}{h}) = \frac{1}{|\Lambda|}\sum_{\lambda\in \Lambda^{\circ}}\rin{L^2}{h}{\pi(\lambda^{\circ})g}\pi(\lambda^{\circ}),
    \end{align}
    which is the well-known \emph{Janssen's representation} \cite[Corollary 9.4.5]{Gr01} extended to functions coming from the Heisenberg modules \cite[Proposition 3.18]{AuEn20}.
\end{remark}
\subsection{Quantum Harmonic Analysis}
We define the \emph{Schatten classes} $\mathcal{S}^p$ as the operator spaces with $p$-summable singular values. In particular, $\mathcal{S}^1$ corresponds to the trace class operators where
\begin{align*}
    \|S\|_{\mathcal{S}^1} := \mathrm{tr}(|S|) = \sum_{n\in \mathbb{N}} \langle |S| e_n, e_n \rangle_{L^2}
\end{align*}
is finite for any orthonormal basis $\{e_n\}_{n\in \mathbb{N}}$, and the Hilbert-Schmidt operators $\mathcal{HS}:=\mathcal{S}^2$ are those operators for which 
\begin{align*}
    \|S\|_{\mathcal{HS}} := \sum_{n\in \mathbb{N}} \langle S e_n, S e_n \rangle_{L^2}
\end{align*}
is finite. Quantum Harmonic Analysis was introduced in \cite{We84}, wherein convolutions and Fourier transforms were extended to operators. Central to these notions is the insight that defining the operator translation
\begin{align*}
    \alpha_z (S) := \pi(z)S\pi(z)^*
\end{align*}
for some $S\in\mathcal{HS}$ corresponds to a translation by $z$ of the Weyl symbol of $S$, that is to say
\begin{align}\label{weyl-shift-covariance}
    \sigma_{\alpha_z(S)} = T_z\sigma_S.
\end{align}
For operators $S,T\in\mathcal{S}^1$, and $f\in L^1(\mathbb{R}^{2d})$, operator-operator and function--operator convolutions can then be defined as
\begin{align*}
    S \star T(z) &:= \mathrm{tr}(S\alpha_z(\Check{T})) \\
    f \star S &:= \int_{\mathbb{R}^{2d}} f(z)\alpha_z(S)\, \d z,
\end{align*}
where $\Check{T}:=PTP$. Furthermore, the integral above can be interpreted in a weak and pointwise sense as in \cite{LuSk18}, or even in the strong sense since $z\mapsto f(z)\alpha_z(S)\psi$ is Bochner integrable \cite{Di14} for all $\psi\in L^2(\mathbb{R}^d)$, as noted in \cite{LuSk19}. Note that the convolution of two operators gives a function on phase space, while the convolution of an operator and a function gives an operator. It turns out that the resulting function and operator in the definition are in $L^1(\mathbb{R}^{2d})$ and the trace class $\mathcal{S}^1$ respectively, and the space of arguments can be extended according to a Young's type relation \cite{LuSk18}. The convolutions correspond to convolutions on the Weyl symbols in the following sense:
\begin{align*}
    S \star T &= \sigma_S * \sigma_T \\
    \sigma_{f\star S} &= f * \sigma_S.
\end{align*}

Along with convolutions, a Fourier transform for operators is introduced. For a trace class operator $S$, the \emph{Fourier-Wigner transform} is defined as 
\begin{align*}
    \mathcal{F}_W (S)(z) := e^{-i\pi x\cdot \omega}\mathrm{tr}(\pi(-z)S).
\end{align*}
The Fourier-Wigner transform of an operator is a function on phase space, and its inverse is the integrated \emph{Schrödinger representation} \cite{LuSk18}, which maps functions on phase space to operators. The Fourier-Wigner transform of an operator is closely related to its Weyl symbol:
\begin{align}\label{weylspreadidentity}
    \mathcal{F}_W (S) = \mathcal{F}_{\Omega} (\sigma_S).
\end{align}
Extending the Fourier-Wigner transform to Hilbert-Schmidt operators gives a unitary transformation from $\mathcal{HS}$ to $L^2(\mathbb{R}^{2d})$. The aforementioned convolutions and Fourier-Wigner transform follow an analogous convolution formula as one would expect based on the function case, although with respect to the symplectic Fourier transform since the convolutions are defined on phase space:
\begin{align*}
    \mathcal{F}_{\Omega} (S\star T) &= \mathcal{F}_W (S)\mathcal{F}_W (T) \\
    \mathcal{F}_W (f\star S) &= \mathcal{F}_{\Omega}(f)\mathcal{F}_W (S).
\end{align*}
In \cite{LuMc24}, modulation spaces of operators were defined, where in the case $p=q$, the space $\mathcal{M}^p$ corresponds to the operators with Weyl symbol in $M^p(\mathbb{R}^{2d})$. The Gelfand triple $$(M^1,L^2,M^\infty)(\mathbb{R}^{d})$$ then corresponds to the operator Gelfand triple $$(\mathcal{M}^1,\mathcal{HS},\mathcal{M}^\infty).$$
Operators in the space $\mathcal{M}^1$ correspond to the nuclear operators $\mathcal{N}(M^{\infty}(\mathbb{R}^d),M^1(\mathbb{R}^d))$, while the space $\mathcal{M}^\infty$ are precisely the bounded operators $\mathcal{L}(M^1(\mathbb{R}^d),M^\infty (\mathbb{R}^d))$ \cite{FeJa22}. Note that $\mathcal{M}^{\infty}$ and $\mathcal{M}^1$ forms a dual pair, and the duality bracket is given explicitly via Weyl symbols: $$\rin{\mathcal{M}^{\infty},\mathcal{M}^1}{T}{S}:=\rin{M^{\infty},M^1}{\sigma_T}{\sigma_S}.$$ Since the modulation spaces $M^p(\mathbb{R}^{2d})$ are invariant under metaplectic transformations, if the Weyl symbol of an operator is in some modulation space; $\sigma_S \in M^p(\mathbb{R}^{2d})$, then the same is true of the Fourier-Wigner transform of the operator, and the integral kernel of the operator \cite{FuSh24}.

By duality we can extend operator translations $\alpha_z$ to  $\mathcal{M}^\infty$. The $\Lambda$-translation invariant operators \cite{FeKo98, Sk21} are then those operators $T\in\mathcal{M}^{\infty}$ such that
\begin{align}\label{lambda-translation-invariant}
    T=\alpha_\lambda (T)
\end{align}
for every $\lambda\in\Lambda$. The canonical example of $\Lambda$-translation invariant operators are the Gabor frame operators:
\begin{example}
    Given some atom $g,h\in L^2(\mathbb{R}^d)$, the Gabor frame operator can be written as an operator-periodisation:
    \begin{align}\label{gab-frame-by-perodization}
        S_{g,h,\Lambda} =\sum_{\lambda\in \Lambda}\alpha_{\lambda}(g\otimes h) =\sum_{\lambda\in\Lambda} \pi(\lambda) g \otimes \pi(\lambda)h
    \end{align}
    and hence, is a $\Lambda$-translation invariant operator.
\end{example}
\noindent By \cref{weyl-shift-covariance}, these are operators with $\Lambda$-periodic Weyl symbols, and so there exists a symplectic Fourier-series type expansion of $\Lambda$-translation invariant operators:
\begin{theorem}[{\cite{FeKo98,Sk21}}]\label{operator-sfs}
    For each $\Lambda$-translation invariant operator $T\in \mathcal{M}^{\infty}$ such that $\sigma_T \in \big(M^1,L^2,M^\infty\big)(\mathbb{R}^{2d}/\Lambda)$, there exists a unique $\mathbf{k} =\{k_{\lambda^{\circ}}\}_{\lambda^{\circ}\in \Lambda^{\circ}}\in  \big(\ell^1,\ell^2,\ell^{\infty}\big)(\Lambda^{\circ})$ such that $\sigma_T \mapsto \mathbf{k}$ is a Banach Gelfand triple isomorphism and
    \begin{align}\label{lambda-trans-inv}
        T = \frac{1}{|\Lambda|}\sum_{\lambda^{\circ}\in \Lambda^{\circ}}k_{\lambda^{\circ}} e^{-\pi i \lambda_1^{\circ} \cdot \lambda_2^{\circ}} \pi(\lambda^{\circ}),
    \end{align}
    where $T$ in general converges in the weak-$*$ sense (using the dual bracket $\rin{\mathcal{M}^{\infty},\mathcal{M}^1}{\cdot}{\cdot}$). In the special case $\mathbf{k}\in \ell^1(\Lambda^{\circ}),$ $T$ converges absolutely and extends to a map in $\mathcal{L}(L^2(\mathbb{R}^d)).$ Lastly, the map $\mathbf{k}\mapsto 
 \sum_{\lambda^{\circ}\in \Lambda^{\circ}}k_{\lambda^{\circ}}e^{-\pi i \lambda^{\circ}_1\cdot \lambda^{\circ}_2}\pi(\lambda^{\circ})$ defines a weak$*$-to-weak-$*$ continuous map from $\ell^{\infty}(\Lambda^{\circ})$ to $\mathcal{M}^{\infty}.$
\end{theorem}
\begin{remark}
    In keeping with the interpretation that $\Lambda$-translation invariant operators are $\Lambda$-periodic, we refer to the unique sequence $\mathbf{k}=\{k_{\lambda^{\circ}}\}_{\lambda^{\circ}\in \Lambda^{\circ}}$ in representation \cref{lambda-trans-inv} the \emph{operator Fourier-coefficients of} $T.$
\end{remark}
$\Lambda$-translation invariant operators observe a pleasing symbol calculus, namely if $\frac{1}{ab}= n$ and $S,T$ are $\Lambda$-translation invariant operators where $\Lambda=a\mathbb{Z}^d\times b\mathbb{Z}^d$, then 
\begin{align}\label{trans-inv-calculus}
    \sigma_S\cdot\sigma_T = \sigma_{S T}.
\end{align}

\section{\texorpdfstring{$\Lambda$}{Lambda}-Translation Invariant Operators via Finite-Rank Operators Generated by the Heisenberg Module}\label{lambda-trans-adj}
In \cite{Sk21} a Janssen-type representation formula was found for general trace-class operators, from which one can then prove a correspondence between $\Lambda$-translation invariant operators (as they should be amenable to an operator-periodisation representation) $T\in \mathcal{M}^{\infty}$ and their operator Fourier coefficients $\mathbf{k}$ as in Theorem \ref{operator-sfs}. We further investigate the theory of $\Lambda$-translation invariant operators through a combination of this Janssen-type correspondence and the characterisation of adjointable maps of the Heisenberg modules as Gabor frame operators. We shall then make the case here that the Gabor frame operators not only give us `canonical examples' of $\Lambda$-translation invariant operators, but  they can be used to topologically  characterise all $\Lambda$-translation invariant operators not only in $\mathcal{M}^{\infty}$, but in $\mathcal{L}\left(L^2(\mathbb{R}^d)\right)$ as well. As a corollary of \cref{gab-frame-by-perodization}, we must be able to characterise all $\Lambda$-translation invariant operators in $\mathcal{M}^{\infty}$ and $\mathcal{L}\left(L^2(\mathbb{R}^d)\right)$ in terms of some topological limit of  operator-periodisation  of finite-rank operators. 

We first start with making the connection between $\Lambda$-translation invariant operators and $A$-adjointable maps of the Heisenberg module $\mathcal{E}_{\Lambda}(\mathbb{R}^d)$ explicit. Let $T\in \mathcal{L}_A(\mathcal{E}_{\Lambda}(\mathbb{R}^d)),$ then it follows that for all $\mathbf{a}\in \ell^1(\Lambda)$ and $f\in M^1(\mathbb{R}^d)$ that:
\begin{align*}
    T(\mathbf{a}f) &= \mathbf{a}(Tf) \\
   \iff  T\left( \sum_{\lambda\in \Lambda}a(\lambda)\pi(\lambda)f \right) &= \sum_{\lambda\in \Lambda}a(\lambda)\pi(\lambda) (Tf)
\end{align*}
If we take $\mathbf{a} = \delta_{\lambda,0}$ in particular, then $T\pi(\lambda)f = \pi(\lambda)(Tf)$ for all $f\in M^1(\mathbb{R}^d)$. If we combine this with the fact that $\mathcal{E}_{\Lambda}(\mathbb{R}^d)\hookrightarrow L^2(\mathbb{R}^d)\hookrightarrow M^{\infty}(\mathbb{R}^d)$, then we obtain that $T_{|M^1(\mathbb{R}^d)}\in \mathcal{M}^{\infty}$ and is a $\Lambda$-translation invariant operator. Subsequently, we may ask, where exactly do the operator Fourier-coefficients $\mathbf{k}$ of $T_{|{M}^1(\mathbb{R}^d)}$ lie within $\ell^{\infty}(\Lambda^{\circ})$? To answer this, we have the following lemma regarding the $A$-adjointable maps of $\mathcal{E}_{\Lambda}(\mathbb{R}^d)$.
\begin{lemma}\label{adjointables-are-mix-gab}
    There exists an $n\in \mathbb{N}$ such that for any $T\in \mathcal{L}_A(\mathcal{E}_{\Lambda}(\mathbb{R}^d))$, we have $T=\sum_{i=1}^n(S_{g_i,h_i,\Lambda})_{|\mathcal{E}_{\Lambda}(\mathbb{R}^d)}$, where $g_1,...,g_n,h_1,...,h_n\in \mathcal{E}_{\Lambda}(\mathbb{R}^d).$ 
\end{lemma}
\begin{proof}
As previously noted, the coefficient $C^*$-algebra $B$ is unital. Consider the ideal $I = \op{span}\{\rin{B}{g}{h}:g,h \in \mathcal{E}_{\Lambda}(\mathbb{R}^d)\}$ of $B$. We claim that $\overline{I}=I$. It follows from fullness of $\mathcal{E}_{\Lambda}(\mathbb{R}^d)$ that $B=\overline{I}$. Since $1_B \in \overline{I}$, and $1_B\in \op{Inv}(B)$, where $\op{Inv}(B)$ is the open set of invertible elements of $B$, then we know that $I \cap \op{Inv}(B) \neq \emptyset$, therefore there exists an invertible element in $I$, say $b\in I$. However, this means $b b^{-1} =1_B \in I$, but $1_B\in I$ implies $I=B.$ Therefore, there must exist $g_1,...,g_n,\psi_1,...,\psi_n \in \mathcal{E}_{\Lambda}(\mathbb{R}^d)$ such that $\sum_{i=1}^n \rin{B}{g_i}{\psi_i}=1_B\in B$. It follows that if $T\in \mathcal{L}_{A}(\mathcal{E}_{\Lambda}(\mathbb{R}^d))$ and $f\in \mathcal{E}_{\Lambda}(\mathbb{R}^d)$, we have $Tf = T(f 1_B) = T\left( \sum_{i=1}^n f\rin{B}{g_i}{\psi_i} \right) = T\left(\sum_{i=1}^n \lin{A}{f}{g_i}\psi_i  \right) = \sum_{i=1}^n \lin{A}{f}{g_i}(T\psi_i) = \sum_{i=1}^n S_{g_i,h_i,\Lambda}f$ where we took $h_i = T\psi_i.$ 
\end{proof}
\begin{remark}
    Notice the uniformity of $n\in \mathbb{N}$ in the proof of Lemma \ref{adjointables-are-mix-gab}. This number actually depends on the density of the lattice $\Lambda$, we discuss this in Section \ref{lattice-discussion}. This observation will be crucial in our forthcoming results.
\end{remark}
We need to introduce another $2$-cocyle so that we can make sense of the connection between the adjointable maps and the canonical representation \cref{lambda-trans-inv} of $\Lambda$-translation invariant operators. We define $c': \mathbb{R}^{2d}\times \mathbb{T}^{2d}\to \mathbb{T}$ via:
\begin{align*}
    c'((x_1,\omega_1),(x_2,\omega_2)) = e^{\pi i (x_2\cdot \omega_1 - x_1\cdot \omega_2)}.
\end{align*}
We denote by $B_1$ the enveloping $C^*$-algebra obtained from the Banach $*$-algebra $\ell^1(\Lambda^{\circ},c')$.
\begin{lemma}
    Let $\rho: \Lambda^{\circ}\to \mathbb{T}$ via $\rho(\lambda^{\circ}) = e^{\pi i \lambda_1^{\circ}\lambda_2^{\circ}}$. Then the map $\widetilde{\rho}: \ell^1(\Lambda^{\circ})\to \ell^1(\Lambda^{\circ})$
    \begin{align}
        \widetilde{\rho}(\mathbf{b})(\lambda^{\circ})=b(-\lambda^{\circ})\rho(\lambda) = b(-\lambda^{\circ})e^{\pi i \lambda_1^{\circ}\cdot \lambda_2^{\circ}}, \qquad \forall \lambda^{\circ}\in \Lambda^{\circ}
    \end{align}
    extends to a $C^*$-isomorphism $\widetilde{\rho}:B_1 \to B.$
\end{lemma}
\begin{proof}
    $\widetilde{\rho}: \ell^1(\Lambda^{\circ})\to \ell^1(\Lambda^{\circ})$ is obviously a Banach space automorphism with inverse 
    \begin{align*}
        (\widetilde{\rho})^{-1}(\mathbf{b})(\lambda^{\circ}) = b(-\lambda^{\circ})e^{-\pi i \lambda^{\circ}_1 \cdot \lambda^{\circ}_2}.
    \end{align*}
    It also follows from a straight-forward computation that: $\widetilde{\rho}(\mathbf{b}_1*_{c'}\mathbf{b}_2) = \widetilde{\rho}(\mathbf{b}_1)*_{\overline{c}} \widetilde{\rho}(\mathbf{b}_2)$ and $\widetilde{\rho}(\mathbf{b}_1^{*_{c'}})=\widetilde{\rho}(\mathbf{b}_1)^{*_{\overline{c}}}$, for each $\mathbf{b}_1,\mathbf{b}_2\in \ell^1(\Lambda^{\circ},c').$ Therefore on the common dense subspace $\ell^1(\Lambda^{\circ})$, we have that $\widetilde{\rho}:\ell^1(\Lambda^{\circ},c')\to \ell^1(\Lambda^{\circ},\overline{c})$ is a Banach $*$-algebra isomorphism, hence $\widetilde{\rho}$ must extend to a $C^*$-isomorphism.
\end{proof}
\begin{remark}
    The isomorphism $B_1\cong B$ can be explained by the fact that the $2$-cocycles $\overline{c}$ and $c'$ are \emph{cohomologous}. An account of cohomology on $2$-cocycles with applications in Gabor analysis can be found in \cite{GeLaLu24}.
\end{remark}
We now reiterate that we have the following continuous embeddings with $p> 2$:
\begin{align}\label{norm-dense-b-embeds}
    \ell^1(\Lambda^{\circ})\hookrightarrow B \cong B_1 \hookrightarrow \ell^2(\Lambda^{\circ})\hookrightarrow\ell^p(\Lambda^{\circ})\hookrightarrow \ell^{\infty}(\Lambda^{\circ}).
\end{align}
All of the inclusions above are norm-dense, except for the last one. Instead, note that since $\ell^1(\Lambda^{\circ})$ is weak-$*$ dense in $\ell^{\infty}(\Lambda^{\circ})$, then all the intermediate spaces between $\ell^1(\Lambda^{\circ})$ and $\ell^{\infty}(\Lambda^{\circ})$ also embed weak-$*$ densely to $\ell^{\infty}(\Lambda^{\circ}).$ We will also use the fact that $\ell^1(\Lambda^{\circ})$ is norm-dense in $\ell^p(\Lambda^{\circ})$ for all $p\geq 1$, so $B\cong B_1$ is also norm-dense in $\ell^p(\Lambda^{\circ})$ for $p\geq 2$ in particular.
\begin{theorem}\label{lambda-trans-a-adjointable}
    $T\in \mathcal{M}^{\infty}$ is a $\Lambda$-translation invariant operator whose Fourier coeffients $\mathbf{k}$ lie in $B_1$, equivalently $\widetilde{\rho}(\mathbf{k})\in B$, if and only if $T$ extends to an $A$-adjointable map. In this case, there exists an $n\in \mathbb{N}$ independent of $T$, such that $T$ is a finite sum of $n$ Gabor frame operators with windows in $\mathcal{E}_{\Lambda}(\mathbb{R}^d)$, and $T$ extends to a map in $\mathcal{L}\left(L^2(\mathbb{R}^d)\right)$ with $\|T\|_{\mathcal{L}(L^2)}=\|\mathbf{k}\|_{B_1}=\|\widetilde{\rho}(\mathbf{k})\|_{B}.$
\end{theorem}
\begin{proof}
    Suppose that $T\in \mathcal{M}^{\infty}$ is a $\Lambda$-translation invariant operator with Fourier coefficients $\mathbf{k} = \{k_{\lambda^{\circ}}\}_{\lambda^{\circ}\in \Lambda^{\circ}}\in B_1$ so that
    \begin{align*}
        T = \frac{1}{|\Lambda|}\sum_{\lambda^{\circ}\in \Lambda^{\circ}} k_{\lambda^{\circ}}e^{- \pi i \lambda^{\circ}_1\cdot \lambda^{\circ}_2}\pi(\lambda^{\circ}) = \frac{1}{|\Lambda|}\sum_{\lambda^{\circ}\in \Lambda^{\circ}}k_{-\lambda^{\circ}}e^{\pi i \lambda_1^{\circ}\cdot \lambda_2^{\circ}}\pi^*(\lambda^{\circ})= \frac{1}{|\Lambda|}\sum_{\lambda\in \Lambda}\widetilde{\rho}(\mathbf{k})(\lambda^{\circ})\pi^*(\lambda^{\circ}).
    \end{align*}
    We then obtain that $T = \overline{\pi}^*_B(\widetilde{\rho}(\mathbf{k}))_{|M^1(\mathbb{R}^d)}$. Now it follows from the fact that $B$ is unital and $\mathcal{E}_{\Lambda}(\mathbb{R}^d)$ is full with respect to $B$, that there exists an $n\in \mathbb{N}_{>0}$ independent from $\mathbf{k}$ (see proof of Lemma \ref{adjointables-are-mix-gab}), and elements $g_1,...,g_n,h_1,...,h_n\in \mathcal{E}_{\Lambda}(\mathbb{R}^d)$ such that $\widetilde{\rho}(\mathbf{k}) = \sum_{i=1}^n \rin{\Lambda^{\circ}}{g_i}{h_i}$. We obtain $T = \sum_{i=1}^n\overline{\pi}^*_B\left(\rin{\Lambda^{\circ}}{g_i}{h_i} \right)_{|M^1(\mathbb{R}^d)}$. It now follows from the Janssen's representation for the Heisenberg modules \cref{janssen-for-heis} that $T = \sum_{i=1}^n (S_{g_i,h_i,\Lambda})_{|M^1(\mathbb{R}^d)}$, from which we can conclude that $T$ extends to an $A$-adjointable map via  Proposition \ref{gab-mix-adjointable}.

    The converse easily follow from the uniqueness of the operator Fourier coefficients, Lemma \ref{lambda-trans-inv}, Lemma \ref{adjointables-are-mix-gab}, and the Janssen's representation for the Heisenberg modules \cref{janssen-for-heis}. In particular, $T = \sum_{i=1}^n S_{g_i,h_i,\Lambda} = \overline{\pi}^*_B\left(\sum_{i=1}^n\rin{B}{g_i}{h_i}\right)$ for some $g_1,...,g_n,h_1,...,h_n\in \mathcal{E}_{\Lambda}(\mathbb{R}^d)$ and $\widetilde{\rho}(\mathbf{k})= \sum_{i=1}^n \rin{B}{g_i}{h_i}$. It follows from Proposition \ref{gab-mix-adjointable} that $T$ extends to $\mathcal{L}\left(L^2(\mathbb{R}^d)\right)$, and that $\|T\|_{\mathcal{L}(L^2)} = \|\overline{\pi}^*_B(\sum_{i=1}^n \rin{B}{g_i}{h_i})\|_{\mathcal{L}(L^2)}=\|\widetilde{\rho}(\mathbf{k})\|_{B}=\|\mathbf{k}\|_{B_1}$. 
\end{proof}
\begin{remark}\label{remark: on l1-case}
    We may work out a proof similar to that of Theorem \ref{lambda-trans-a-adjointable} above and obtain a statement that is consistent with the embedding $\ell^1(\Lambda^{\circ})\hookrightarrow B$ and $M^1(\mathbb{R}^d)\hookrightarrow \mathcal{E}_{\Lambda}(\mathbb{R}^d)$, that is: $T\in \mathcal{M}^{\infty}$ with operator Fourier coefficients $\mathbf{k}\in \ell^1(\Lambda^{\circ})$ if and only if $T = \sum_{i=1}^n (S_{g_i,h_i, \Lambda})_{|M^1(\mathbb{R}^d)}$ with $g_1,...,g_n,h_1,...,h_n\in M^1(\mathbb{R}^d).$
\end{remark}
\begin{lemma}\label{est-for-periodic}    For each $\Lambda$-translation invariant operator $T\in \mathcal{M}^{\infty}$ with Fourier coefficients $\mathbf{k}\in \ell^{\infty}(\Lambda^{\circ})$, we have the following estimate:
    \begin{align*}
        \|T\|_{\mathcal{M}^{\infty}} \leq \frac{\|\mathbf{k}\|_{\infty}}{|\Lambda|}.
    \end{align*}
    As a corollary, $\displaystyle \|T\|_{\mathcal{M}^{\infty}}\leq \frac{\|\mathbf{k}\|_{\ell^p(\Lambda^{\circ})}}{|\Lambda|^{1-\frac{1}{p}}}$ whenever $\mathbf{k}\in \ell^p(\Lambda^{\circ})$ for $1\leq p < \infty.$
\end{lemma}
\begin{proof}
    We compute $\|T\|_{\mathcal{M}^{\infty}} =\sup\left\{\left|\rin{M^{\infty},M^1}{Tf}{g} 
    \right| : \|f\|_{M^1}=\|g\|_{M^1}=1 \right\}.$ Using \cref{lambda-trans-inv} we obtain:
    \begin{align*}
        \|T\|_{\mathcal{M}^{\infty}}&\leq \frac{1}{|\Lambda|}\|\mathbf{k}\|_{\infty} \sup\left\{\sum_{\lambda^{\circ}\in \Lambda^{\circ}}|\mathcal{V}_fg(\lambda^{\circ})|: \|f\|_{M^1}=\|g\|_{M^1}=1\right\}\\
        &\leq \frac{1}{|\Lambda|}\|\mathbf{k}\|_{\infty} \sup\left\{\int_{\mathbb{R}^{2d}}|\mathcal{V}_fg(x,\omega)|\d x\d \omega: \|f\|_{M^1}=\|g\|_{M^1}=1\right\} \\
        &\leq \frac{1}{|\Lambda|} \|\mathbf{k}\|_{\infty},
    \end{align*}
    where the last inequality follow from \cite[Proposition 4.10]{Ja18}. Finally, if $\mathbf{k}\in \ell^p(\Lambda^{\circ})$, then it follows from $\|\mathbf{k}\|_{\infty}\leq \left(\sum_{\lambda^{\circ}\in \Lambda^{\circ}}|k_{\lambda^{\circ}}|^p\right)^{1/p}$ and \cref{lp-for-adjoint-lattice} that $\displaystyle \|T\|_{\mathcal{M}^{\infty}}\leq \frac{\|\mathbf{k}\|_{\ell^p(\Lambda^{\circ})}}{|\Lambda|^{1-\frac{1}{p}}}$.
\end{proof}
\noindent 
Here we have our main charaterisation theorem of $\Lambda$-translation invariant operators in $\mathcal{M}^{\infty}$.
\begin{theorem}\label{lamb-trans-operator-limit}
    There exists an $n\in \mathbb{N}$ such that every $\Lambda$-translation invariant operator $T\in \mathcal{M}^{\infty}$ is the weak-$*$ limit of operator-periodisations of rank-$n$ operators generated by functions coming from $\mathcal{E}_{\Lambda}(\mathbb{R}^d).$ In particular, if the operator Fourier coefficients $\mathbf{k}$ of $T$ belongs in $\ell^p(\Lambda^{\circ})$ with $p\geq 2$ (resp. $1\leq p <2$), then $T$ is the $\mathcal{M}^{\infty}$-norm limit of operator-periodisations of rank-$n$ operators generated by functions coming from $\mathcal{E}_{\Lambda}(\mathbb{R}^d)$ (resp. $M^1(\mathbb{R}^d)$). 
\end{theorem}
\begin{proof}
    Fix $n\in \mathbb{N}$ as in Theorem \ref{lambda-trans-a-adjointable}. Suppose $T$ is a $\Lambda$-translation invariant operator in $\mathcal{M}^{\infty}$ with Fourier coefficients $\mathbf{k}\in \ell^{\infty}(\Lambda^{\circ}).$ Let $\{\mathbf{k}_i\}_{i\in I}\subseteq B_1$ be a net such that $\mathbf{k}_i\to \mathbf{k}$ in the weak-$*$ topology. For each $i\in I$, define the $\Lambda$-translation invariant operator $T_i =\frac{1}{|\Lambda|} \sum_{\lambda^{\circ}\in \Lambda^{\circ}}k_{i,\lambda^{\circ}}e^{-\pi i \lambda_1^{\circ}\cdot\lambda_2^{\circ}}\pi(\lambda^{\circ})$. It follows from the last sentence of Theorem \ref{lambda-trans-inv} that $T_i \to T$ in the weak-$*$ in $\mathcal{M}^{\infty}$, furthermore, Theorem \ref{lambda-trans-a-adjointable} says that each $T_i$ is a finite sum of $n$ Gabor frame operators with windows from $\mathcal{E}_{\Lambda}(\mathbb{R}^d)$. It follows from \cref{gab-frame-by-perodization} that each $T_i$ is the operator-periodisation of rank-$n$ operators with generators coming from $\mathcal{E}_{\Lambda}(\mathbb{R}^d)$.
    
    If $T$ has Fourier coefficients $\mathbf{k}\in \ell^p(\Lambda^{\circ})$ with $p\geq 2$, then the norm-dense embeddings \eqref{norm-dense-b-embeds} assure us that there exists a sequence $\{\mathbf{k}_m\}_{m\in \mathbb{N}}\subseteq B_1$ such that $\mathbf{k}_m\to \mathbf{k}$ in $\ell^p(\Lambda^{\circ})$-norm. Define for each $m\in \mathbb{N}$ the $\Lambda$-translation invariant operator $T_m = \frac{1}{|\Lambda|}\sum_{\lambda^{\circ}\in \Lambda^{\circ}}k_{m,\lambda^{\circ}}e^{-\pi i \lambda_1^{\circ}\cdot \lambda_2^{\circ}}\pi(\lambda^{\circ})$. It follows from Lemma \ref{est-for-periodic} that $\displaystyle \|T-T_m\|_{\mathcal{M}^{\infty}}\leq \frac{\|\mathbf{k}-\mathbf{k}_m\|_{\ell^p(\Lambda^{\circ})}}{|\Lambda|^{1-\frac{1}{p} }}$. Since $\|\mathbf{k}-\mathbf{k}_m\|_{\ell^p(\Lambda^{\circ})}\to 0$ and, each $T_m$ is the operator-periodisation of rank-$n$ operators with generators coming from $\mathcal{E}_{\Lambda}(\mathbb{R}^d)$, we are done. The proof for $1\leq p<2$ is similar, following Remark \ref{remark: on l1-case}
\end{proof}
We have seen that the $C^*$-algebraic aspects of the theory coming from the Heisenberg modules has allowed us to fully characterise $\Lambda$-translation invariant operators in $\mathcal{M}^{\infty}$ where we have gone outside the usual confines of the Banach Gelfand triple setting $(\ell^1,\ell^2,\ell^{\infty})(\Lambda^{\circ})$ by considering the intermediate sequence spaces $B$ and $B_1$. We now consider the von Neumann algebraic aspect of the theory by going out of the the $\mathcal{M}^{\infty}$ setting. We denote the space of $\Lambda$-invariant operators in $\mathcal{L}\left(L^2(\mathbb{R}^d)\right)$ via:
\begin{align*}
\mathcal{L}_{\Lambda}\left(L^2(\mathbb{R}^d)\right):=\left\{T\in \mathcal{L}\left(L^2(\mathbb{R}^d)\right): T=\alpha_{\lambda}(T), \forall \lambda\in \Lambda \right\}.
\end{align*}
The following lemma is immediate.
\begin{lemma}\label{lamb-inv-in-L^2}
    $T\in \mathcal{L}_{\Lambda}\left(L^2(\mathbb{R}^d)\right)$ if it is an extension of a $\Lambda$-translation invariant operator in $S\in \mathcal{M}^{\infty}$ such that $\op{Im}(S)\subseteq L^2(\mathbb{R}^d)$. In this case, $T$ can be represented as in \cref{lambda-trans-inv}, with operator Fourier-coefficients $\mathbf{k}\in \ell^2(\Lambda^{\circ}).$
\end{lemma}
\begin{proof}
    $T$ restricts in $M^1(\mathbb{R}^d)$ to a $\Lambda$-translation invariant operator in $\mathcal{M}^{\infty}$, so the first statement is immediate. We only show that $T\in \mathcal{L}_{\Lambda}\left(L^2(\mathbb{R}^d)\right)$ will necessarily have operator Fourier-coefficients satisfying $\mathbf{k}\in \ell^2(\Lambda^{\circ}).$ We have from \cref{lambda-trans-inv} that for all $f,g\in L^2(\mathbb{R}^d)$
    \begin{align*}
        \rin{L^2}{Tf}{Tg} = \frac{1}{|\Lambda|^2}\sum_{\lambda^{\circ},\nu^{\circ}\in \Lambda^{\circ}}k_{\lambda^{\circ}}\overline{k_{\nu^{\circ}}}e^{-\pi i (\lambda_1^{\circ}\cdot \lambda_2^{\circ} - \nu_1^{\circ}\cdot \nu_2^{\circ})} \rin{L^2}{\pi(\lambda^{\circ})f}{\pi(\lambda^{\circ})g}.
    \end{align*}
    By letting $f=g$, and taking the supremum over all $\|f\|_{L^2}=1$ in the equation above, we obtain $\infty > \|T\|_{\mathcal{L}(L^2)}^2\geq \frac{1}{|\Lambda|^2}\sum_{\lambda^{\circ}\in \Lambda^{\circ}}|k_{\lambda^{\circ}}|^2$, which shows $\mathbf{k}\in \ell^2(\Lambda^{\circ}).$
\end{proof}
\begin{remark}
    Note that all of our results on $\Lambda$-translation invariant operators in this section have an analogue for when we replace $\Lambda$ with its adjoint $\Lambda^{\circ}.$ Relevant to this is the fact that we can construct the Heisenberg module again by switching the roles of the lattice $\Lambda$ and its lattice $\Lambda^{\circ}$, giving us $\mathcal{E}_{\Lambda^{\circ}}(\mathbb{R}^d)$. It is in fact also true that $\mathcal{E}_{\Lambda^{\circ}}(\mathbb{R}^d)=\mathcal{E}_{\Lambda}(\mathbb{R}^d)$ \cite[Proposition 3.17]{AuEn20}.
\end{remark}

Next, if $S\subseteq \mathcal{L}\left(L^2(\mathbb{R}^d)\right)$, we denote the \emph{commutant} of $S$ in $\mathcal{L}\left(L^2(\mathbb{R}^d)\right)$ via $S':= \{T\in \mathcal{L}\left(L^2(\mathbb{R}^d)\right): Ts=sT, \ \forall s\in S\}$. The \emph{double commutant} of $S$ is $S''=(S')'$. We use $\overline{S}, \overline{S^{\op{SOT}}}, \overline{S^{\op{WOT}}}$ to denote the closure of $S$ with respect to the norm, strong-operator, and weak-operator topologies respectively. Finally, von Neumann's double commutant theorem \cite[Theorem I.7.1]{Da96} will feature prominently in the subsequent proofs.
\begin{lemma}\label{various-equalities-vn}
    We have:
    \begin{align}\label{commutants}
         \overline{\pi}_{B}^*(\ell^1(\Lambda^{\circ}))'' = \overline{\pi}_{B}^*(B)''.
    \end{align}
\end{lemma}
\begin{proof}
    Note that $\overline{\pi}_{B}^*(\ell^1(\Lambda^{\circ}))$ is a unital $*$-subalgebra of $\mathcal{L}\left(L^2(\mathbb{R}^d)\right)$, therefore it follows from the double commutant theorem and the fact that the norm-topology contains the strong operator topology, which further contains the weak operator topology, that:
\begin{align}\label{vn-interm-1}
    \overline{\pi}_{B}^*(\ell^1(\Lambda^{\circ})) \subseteq \overline{\overline{\pi}_{B}^*(\ell^1(\Lambda^{\circ}))} \subseteq \overline{\overline{\pi}_{B}^*(\ell^1(\Lambda^{\circ}))^{\text{WOT}}} = \overline{\overline{\pi}_{B}^*(\ell^1(\Lambda^{\circ}))^{\text{SOT}}} = \overline{\pi}_{B}^*(\ell^1(\Lambda^{\circ}))''.
\end{align}
However, $\overline{\pi}_{B}^*$ is a $*$-(anti)homomorphism, and thus has a closed range, from which it follows that 
\begin{align}\label{vn-interm-2}
\overline{\overline{\pi}_{B}^*(\ell^1(\Lambda^{\circ}))} = \overline{\overline{\pi}_{B}^*(\overline{\ell^1(\Lambda^{\circ})})} = \overline{\pi}_{B}^*(B).    
\end{align}
Therefore Equations \cref{vn-interm-1} and \cref{vn-interm-2} imply that $\overline{\pi}^*_{B}(B)\subseteq \overline{\pi}_{B}^*(\ell^1(\Lambda^{\circ}))''$, from which it follows that $\overline{\pi}_{B}^*(B)'' \subseteq \overline{\pi}_{B}^*(\ell^1(\Lambda^{\circ}))''$. The reverse inclusion follows easily from the fact that $\overline{\pi}_{B}^*(\ell^1(\Lambda^{\circ}))\subseteq \overline{\pi}_{B}^*(B)$, whence $\overline{\pi}_{B}^*(\ell^1(\Lambda^{\circ}))''\subseteq \overline{\pi}_{B}^*(B)''$. We have proved $\overline{\pi}^*_{B}(\ell^1(\Lambda^{\circ}))'' = \overline{\pi}_{B}^*(B)''.$
\end{proof}
\begin{lemma}\label{various-equalities-vn-2}
    We have: 
    \begin{align}\label{commutants-2}
        \mathcal{L}_{\Lambda}\left(L^2(\mathbb{R}^d)\right)= \mathcal{L}_{\Lambda^{\circ}}\left(L^2(\mathbb{R}^d)\right)'
    \end{align}
\end{lemma}
\begin{proof}
    Suppose $T\in \mathcal{L}_{\Lambda^{\circ}}\left(L^2(\mathbb{R}^d)\right)'$, because $\pi(\lambda)\in \mathcal{L}_{\Lambda^{\circ}}\left(L^2(\mathbb{R}^d)\right)$ for all $\lambda\in \Lambda$, then $T$ commutes with all $\pi(\lambda)\in \Lambda$, so $T\in \mathcal{L}_{\Lambda}\left(L^2(\mathbb{R}^d)\right).$ We obtain $\mathcal{L}_{\Lambda^{\circ}}\left(L^2(\mathbb{R}^d)\right)'\subseteq \mathcal{L}_{\Lambda}\left(L^2(\mathbb{R}^d)\right).$ For the reverse inclusion, we use \ref{lamb-inv-in-L^2} adapted for $\Lambda^{\circ}$-translation invariant operators, so that $T^{\circ}\in \mathcal{L}_{\Lambda^{\circ}}\left(L^2(\mathbb{R}^d)\right)$ can always be written 
    \begin{align}\label{vn-interm-3}
    T^{\circ}= \frac{1}{|\Lambda^{\circ}|}\sum_{\lambda\in \Lambda}k_{\lambda}e^{-\pi i \lambda_1\cdot \lambda_2}\pi(\lambda)    
    \end{align}
    for some $\{k_{\lambda}\}_{\lambda\in \Lambda}\in \ell^{\infty}(\Lambda)$. On the other hand, if $T\in \mathcal{L}_{\Lambda}\left(L^2(\mathbb{R}^d)\right)$, we find due to \cref{vn-interm-3} that $TT^{\circ}=T^{\circ}T,$ whence $\mathcal{L}_{\Lambda}\left(L^2(\mathbb{R}^d)\right)\subseteq \mathcal{L}_{\Lambda^{\circ}}\left(L^2(\mathbb{R}^d)\right)'$. We now have $\mathcal{L}_{\Lambda}\left(L^2(\mathbb{R}^d)\right)= \mathcal{L}_{\Lambda^{\circ}}\left(L^2(\mathbb{R}^d)\right)'$
\end{proof}
\noindent We now obtain a von Neumann algebraic version of the Theorem \ref{lamb-trans-operator-limit}.
\begin{theorem}\label{beyond-bgt}
     We have:
    \begin{align}\label{final-vn}
        \mathcal{L}_{\Lambda}\left(L^2(\mathbb{R}^d)\right)=\overline{\pi}^*_B(B)''.
    \end{align}
    As a corollary, there exists an $n\in \mathbb{N}$ such that every operator in $\mathcal{L}_{\Lambda}\left(L^2(\mathbb{R}^d)\right)$ is a weak-operator, or strong-operator limit of periodisations of rank-$n$ operators generated by functions coming from $\mathcal{E}_{\Lambda}(\mathbb{R}^d).$
\end{theorem}
\begin{proof}
Suppose $T\in \overline{\pi}^*_B(\ell^1(\Lambda^{\circ}))'$, because $\pi(\lambda^{\circ})\in \overline{\pi}^*_B(\ell^1(\Lambda^{\circ}))$ for all $\lambda^{\circ}\in \Lambda^{\circ}$, then $T$ commutes $\pi(\lambda^{\circ})$ for all $\lambda\in \Lambda^{\circ}$. Therefore $$\overline{\pi}^*_{B}(\ell^1(\Lambda^{\circ}))'\subseteq \mathcal{L}_{\Lambda^{\circ}}\left(L^2(\mathbb{R}^d)\right).$$ 
Taking the commutant of the inclusion gives us $$\mathcal{L}_{\Lambda}\left(L^2(\mathbb{R}^d)\right)=\mathcal{L}_{\Lambda^{\circ}}\left(L^2(\mathbb{R}^d)\right)'\subseteq \overline{\pi}_{B}^*(\ell^1(\Lambda^{\circ}))''.$$ 
On the other hand, we note that $\pi(\lambda)\in \overline{\pi}_{B}^*(\ell^1(\Lambda^{\circ}))'$ for all $\lambda\in \Lambda$, hence if $T\in \overline{\pi}_{B}^*(\ell^1(\Lambda^{\circ}))''$, then $T\pi(\lambda)=\pi(\lambda)T$ for all $\lambda\in \Lambda$. We have obtained $$\overline{\pi}_{B}^*(\ell^1(\Lambda^{\circ}))'' \subseteq \mathcal{L}_{\Lambda}\left(L^2(\mathbb{R}^d)\right).$$ 
All-in-all, we have shown that $$\mathcal{L}_{\Lambda}\left(L^2(\mathbb{R}^d)\right)= \mathcal{L}_{\Lambda^{\circ}}\left(L^2(\mathbb{R}^d)\right)'=\overline{\pi}_{B}^*(\ell^1(\Lambda^{\circ}))''.$$ \cref{final-vn} now follows from Equations \cref{commutants} and \cref{commutants-2}.

As a corollary of \cref{final-vn} and the double commutant theorem:
    \begin{align*}
        \mathcal{L}_{\Lambda}(L^2(G)) = \overline{\pi}_{B}^*(B)'' = \overline{\overline{\pi}_{B}^*(B)^{\op{SOT}}} = \overline{\overline{\pi}_{B}^*(B)^{\op{WOT}}}.
    \end{align*}
    As a consequence of unitality of $B$ and fullness of $\mathcal{E}_{\Lambda}(\mathbb{R}^d)$, there exists a fixed $n\in \mathbb{N}$ such that the operators in $\overline{\pi}^*_B(B)$ (similar to the proof in Lemma \ref{adjointables-are-mix-gab}) are exactly of the form $\sum_{i=1}^n S_{g_i,h_i,\Lambda}$ where $g_1,...,g_n,h_1,...,h_n\in \mathcal{E}_{\Lambda}(\mathbb{R}^d)$. The result now follows from \cref{gab-frame-by-perodization}.
\end{proof}
\section{Weyl and Spreading function Quantisation Schemes and the Parity Operator}
Before we explicitly define operator modulations, we consider the background motivating this approach, following \cite{Gr76}. For any operator $S\in (\mathcal{M}^1,\mathcal{HS},\mathcal{M}^{\infty})$, there exists a unique \emph{spreading function} $\eta_S \in (M^1,L^2,M^{\infty})(\mathbb{R}^{2d})$, such that
\begin{align}\label{spreadingrep}
    S = \int_{\mathbb{R}^{2d}} \eta_S (z) \pi(z)\, \d z.
\end{align}
The operator-valued integral can be understood weakly as in \cref{contrecon} when $S\in \mathcal{HS}$;  for the case $S\in \mathcal{M}^{\infty}$, see Remark \ref{dual-interpretation} below. The spreading function thus has an intuitive interpretation, as describing how the time-frequency concentration of a function is ``spread" by the operator; a spreading function concentrated near the origin will act similarly to the identity, and the effect of a translation of a spreading function is also clear. The closely related Weyl symbol of an operator is also a Gelfand triple isomorphism $S\mapsto \sigma_S$, and inversion of this mapping leads us to the Weyl quantisation (or Weyl transform) $(M^1,L^2,M^{\infty})(\mathbb{R}^{2d}) \ni \sigma\mapsto L_{\sigma}\in (\mathcal{M}^1,\mathcal{H}S,\mathcal{M}^{\infty})$ defined
\begin{align*}
    \rin{M^{\infty}, M^1}{L_{\sigma}f}{g} = \rin{M^{\infty},M^1}{\sigma}{W(g,f)}
\end{align*}
for all $f,g\in M^1(\mathbb{R}^{2d})$. The spreading function and Weyl symbol of an operator are related by the symplectic Fourier transform, with a phase factor:
\begin{align}\label{spread-weyl-relations}
    \eta_S(z) = e^{-i\pi x\cdot \omega}\mathcal{F}_{\Omega}(\sigma_S)(z).
\end{align}
\begin{remark}\label{dual-interpretation}
     Note that we can rewrite \cref{spread-weyl-relations} as
     \begin{align}\label{spread-by-duality}
         \sigma_S = (\mathcal{F}_{\Omega}\circ \op{Ph})(\sigma_S)
     \end{align}
     where $\op{Ph}: M^1(\mathbb{R}^{2d})\to M^1(\mathbb{R}^{2d})$ is the isomorphism defined by $(\op{Ph}F)(z) = e^{i\pi x\cdot \omega}F(z)$ for $F\in M^1(\mathbb{R}^{2d})$ and $z=(x,\omega)\in \mathbb{R}^{2d}$. By dually extending we obtain $\op{Ph}: M^{\infty}(\mathbb{R}^{2d})\to M^{\infty}(\mathbb{R}^{2d})$ and $\mathcal{F}_{\Omega}:M^{\infty}(\mathbb{R}^{2d})\to M^{\infty}(\mathbb{R}^{2d})$, and so by writing $\eta_S$ using \cref{spread-by-duality}, we know how to interpret $\eta_S$ as a distribution in $M^{\infty}(\mathbb{R}^{2d})$ when $S\in \mathcal{M}^{\infty}$. Note that this also allows us to interpret the integral $S$ in \cref{spreadingrep} as an operator in $\mathcal{M}^{\infty}$ using:
    \begin{align}\label{S-in-M-inf}
        \rin{M^{\infty},M^1}{Sf}{g} = \rin{M^{\infty}M^1}{(\mathcal{F}_{\Omega}\circ\op{Ph})(\eta_S)}{W(g,f)}=\rin{M^{\infty},M^1}{\eta_S}{(\mathcal{F}_{\Omega}\circ \op{Ph}^{-1})W(g,f)}.
    \end{align}
    Observe that when $\eta_S$ is in $M^1(\mathbb{R}^{2d})$ or $L^2(\mathbb{R}^d)$, then one recovers \cref{spreadingrep} from \cref{S-in-M-inf}. For the remainder of the paper, we shall use these \emph{by-duality} extensions as a way to interpret such integrals.
\end{remark}
We now turn our attention to a particular operator, the \emph{parity operator} $P$. The parity operator is defined as the map
\begin{align*}
    P: L^2(\mathbb{R}^d)&\to L^2(\mathbb{R}^d) \\
    f(t)&\mapsto f(-t).
\end{align*}
The parity operator has the property
\begin{align}\label{paritytfshift}
    P\pi(z) = \pi(-z)P,
\end{align}
since
\begin{align*}
    P\pi(z)f(t) &= Pe^{2\pi i \omega t}f(t-x) \\
    &= e^{-2\pi i \omega t}f(-t-x) \\
    &= \pi(-z)Pf(t).
\end{align*}
We then find that
\begin{align}\label{parityintertw}
    \alpha_z(P) &= \pi(z)P\pi(z)^* \nonumber \\
    &= e^{-2\pi i x\cdot\omega} \pi(z)P \pi(-z) \nonumber \\
    &= e^{-2\pi i x\cdot\omega}\pi(z)\pi(z)P \nonumber \\
    &= e^{-4\pi i x\cdot\omega} \pi(2z) P.
\end{align}
The following Lemma will also help with some of our characterisation results in the sequel:
\begin{lemma}\label{parity-isometry}
    The parity operator is a self-inverse unitary map in $\mathcal{L}\left(L^2(\mathbb{R}^d)\right)$, and restricts to an isometric isomorphism $P_{|M^1(\mathbb{R}^d)}:M^1(\mathbb{R}^d)\to M^1(\mathbb{R}^d).$ As a corollary of the latter statement, we can see $P$ as an operator in $\mathcal{M}^{\infty}.$
\end{lemma}
\begin{proof}
    That $P\in \mathcal{L}\left(L^2(\mathbb{R}^d)\right)$ is self-inverse is obvious, while unitarity follows from:
    \begin{align*}
        \int_{\mathbb{R}^{2d}}f(z)\d z = \int_{\mathbb{R}^{2d}}f(-z)\d z, \qquad \forall f\in L^1(\mathbb{R}^{d}).
    \end{align*}
     That $P$ restricts to an actual isometric isomorphism on $M^1(\mathbb{R}^d)$ can be found in \cite[Equation (4.11)]{Ja18}. 
\end{proof}
We claim the following.
\begin{proposition}\label{symplparity}
The integral $\displaystyle \int_{\mathbb{R}^{2d}}e^{2\pi i \Omega(z,z')-i\pi x' \omega'}\pi(z')\d z'$ defines an operator in $\mathcal{L}(L^2(\mathbb{R}^d))$ for each $z\in \mathbb{R}^{2d}$, and
    \begin{align*}
        \int_{\mathbb{R}^{2d}} e^{2\pi i \Omega(z,z') - i \pi x' \omega'}\pi(z') \, \d z' = 2^d \alpha_z(P).
    \end{align*}
Henceforth, we can also see that $\displaystyle e^{2\pi i \Omega(z,z')-i\pi x'\omega'}\pi(z')\d z'$ defines an operator in $\mathcal{M}^{\infty}.$
\end{proposition}
\begin{proof}
For each $z\in \mathbb{R}^{2d}$ and $g\in L^2(\mathbb{R}^d)$, we claim that the map:
\begin{align}\label{integral-functional}
    f\mapsto \int_{\mathbb{R}^{2d}} e^{-2\pi i \Omega(z,z') + i\pi x'\omega'} \rin{L^2}{f}{\pi(z')g}\d z'
\end{align}
defines a continuous linear functional in $L^2(\mathbb{R}^d)$. Linearity is trivial, and boundedness follow from
\begin{align}\label{intermediate-orig-equ}
    \int_{\mathbb{R}^{2d}} e^{-2\pi i \Omega(z,z') + i \pi x' \omega'} \langle f,\pi(z') g\rangle_{L^2}\, \d z'= \mathcal{F}_{\Omega} (A_g f) (z) = W(f,g)(z).
\end{align}
Substituting the identity relating the Wigner distribution and STFT (cf. Lemma 4.3.1, \cite{Gr01}),
\begin{align}\label{intermediate-show-linearfunctional}
    W(f,g)(z) = 2^d e^{4\pi i x \omega} V_{Pg} f(2z),
\end{align}
into \cref{intermediate-orig-equ} then gives
\begin{align}\label{intermediate-equation}
    \int_{\mathbb{R}^{2d}} e^{-2\pi i \Omega(z,z') + i \pi x' \omega'} \langle f,\pi(z') g\rangle_{L^2}\, \d z' =2^d e^{4\pi i x\omega}V_{Pg}f(2z)= 2^d e^{4\pi i x \omega} \langle f,  \pi(2z) Pg \rangle_{L^2}.
\end{align}
It then follows from H\"older's inequality and the fact that $\pi(2z)$ and $P$ are both isometries in $L^2(\mathbb{R}^d)$ that
\begin{align}\label{intermediate-from-holder}
    \left|\int_{\mathbb{R}^{2d}} e^{-2\pi i \Omega(z,z') + i \pi x' \omega'} \langle f,\pi(z') g\rangle_{L^2}\, \d z'\right|\leq 2^d \|f\|_{L^2} \|g\|_{L^2}.
\end{align}
Hence the map \cref{integral-functional} is a continuous linear functional, and it follows from Riesz representation theorem that we have a linear mapping $\displaystyle \int_{\mathbb{R}^{2d}} e^{2\pi i \Omega(z,z')-i\pi x'\omega'}\pi(z')\d z': L^2(\mathbb{R}^d)\to L^2(\mathbb{R}^d)$ for all $z\in \mathbb{R}^{2d}$ where
\begin{align*}
    \rin{L^2}{f}{\int_{\mathbb{R}^{2d}} e^{2\pi i \Omega(z,z')-i\pi x'\omega'}\pi(z')\d z' (g)} &= \int_{\mathbb{R}^{2d}} e^{-2\pi i \Omega(z,z') + i \pi x' \omega'} \langle f,\pi(z') g\rangle_{L^2}\, \d z'.
\end{align*}
The fact that $\displaystyle \int_{\mathbb{R}^{2d}} e^{2\pi i \Omega(z,z')-i\pi x'\omega'}\pi(z')\d z'\in \mathcal{L}(L^2(\mathbb{R}^d))$ follows by taking the supremum of the inequality \cref{intermediate-from-holder} for all $\|f\|_{L^2}=1$,and $\|g\|_{L^2}=1$.
Furthermore, \cref{intermediate-equation} and \cref{parityintertw} implies that for all $f,g\in L^2(\mathbb{R}^d)$
\begin{align*}
    \rin{L^2}{f}{\int_{\mathbb{R}^{2d}} e^{2\pi i \Omega(z,z')-i\pi x'\omega'}\pi(z')\d z' (g)} &= \int_{\mathbb{R}^{2d}} e^{-2\pi i \Omega(z,z') + i \pi x' \omega'} \langle f,\pi(z') g\rangle_{L^2}\, \d z'\\ 
    &= \rin{L^2}{f}{2^d \alpha_{z}(P)g}
\end{align*}
which gives us $\displaystyle \int_{\mathbb{R}^{2d}} e^{2\pi i \Omega(z,z') - i \pi x' \omega'}\pi(z') \, \d z' = 2^d \alpha_z(P).$ From here it follows that the integral also defines an operator in $\mathcal{M}^{\infty}$ because $\alpha_z(P)\in\mathcal{M}^{\infty}$.
\end{proof}
From here we find a spreading-type quantisation of a Weyl symbol directly, without needing to take the symplectic Fourier transform:
\begin{proposition}\label{weylquant}
    Given $\sigma\in M^{\infty}(\mathbb{R}^{2d})$, the Weyl quantisation $L_{\sigma}\in \mathcal{M}^{\infty}$ can be expressed as 
    \begin{align*}
        L_{\sigma} = 2^d \int_{\mathbb{R}^{2d}} \sigma(z) \alpha_{z}(P)\, \d z,
    \end{align*}
    or equivalently 
    \begin{align*}
        L_{\sigma} =  \int_{\mathbb{R}^{2d}} \sigma(\tfrac{z}{2}) e^{-\pi i x \cdot\omega} \pi(z) P\, \d z.
    \end{align*}
\end{proposition}
\begin{proof}
    We know that from the relation between the Weyl symbol and spreading function that 
    \begin{align*}
        L_{\sigma} &= \int_{\mathbb{R}^{2d}} e^{-i\pi x\cdot\omega }\mathcal{F}_{\Omega} (\sigma) (z)\pi(z)\, \d z \\
        &= \int_{\mathbb{R}^{2d}} e^{-i\pi x\cdot\omega} \int_{\mathbb{R}^{2d}}e^{-2\pi i \Omega(z,z')} \sigma(z')\, \d z' \, \pi(z)\, \d z \\
        &= \int_{\mathbb{R}^{2d}} \sigma(z') \int_{\mathbb{R}^{2d}} e^{2\pi i \Omega(z',z) - i\pi x\cdot\omega}  \pi(z)\, \d z\, \d z'.
    \end{align*}
    Inserting the result of \cref{symplparity} then gives
    \begin{align*}
        L_{\sigma} = 2^d \int_{\mathbb{R}^{2d}} \sigma(z') \alpha_{z'}(P) \d z'.
    \end{align*}
    The second form follows from the identity \cref{parityintertw}, along with the coordinate transform $z\mapsto \tfrac{z}{2}$.
    
\end{proof}
The above representation of an operator is also discussed in \cite{Te18}. In this paper we consider operator modulations, and $\Lambda$-modulation invariant operators. The above quantisation intuition serves as a motivation of this approach; while in the case of $\Lambda$-translation invariant operators we have spreading quantisations in the form of \cref{spreadingrep} of functions supported on the lattice $\Lambda$, in this work we find $\Lambda$-modulation invariant operators are precisely those operators which are quantisations in the form \cref{weylquant} of functions supported on the lattice $\Lambda$.

\section{A Modulation for Operators}
To motivate the concept of a modulation for operators, we begin by considering translations of operators. We recall that a translation of an operator is defined by
\begin{align*}
    \alpha_z(S) := \pi(z)S\pi(z)^*,
\end{align*}
and this $\alpha_z$ operation corresponds to a translation of the Weyl symbol. A cornerstone of QHA is the fact that the Fourier-Wigner transform interacts with the symplectic Fourier transform and convolutions analagously to the function case, namely that operator convolutions satisfy Fourier convolution identities with respect to the Fourier-Wigner transform. If we consider the modulation of a function as a translation of its Fourier transform, then one may reasonably assume that a modulation of an operator should be reflected in a translation of the Fourier-Wigner transform of the operator. Motivated by this, we define the \emph{operator modulation} $\beta_w (S)$ as follows:
\begin{definition}
    Let $w\in\mathbb{R}^{2d}$, and $S\in\mathcal{M}^{\infty}$. Then
    \begin{align*}
        \beta_w(S) := e^{-\pi i w_1 w_2/2}\pi\Big(\frac{w}{2}\Big)S\pi\Big(\frac{w}{2}\Big).
    \end{align*}
\end{definition}
This is precisely the operation corresponding to a symplectic modulation of the Weyl symbol:
\begin{proposition}\label{symbolmodulation}
    Given $w\in\mathbb{R}^{2d}$, and $S\in\mathcal{M}^{\infty}$,
    \begin{align*}
        \sigma_{\beta_w(S)} = M_w \sigma_S,
    \end{align*}
    where $M_w$ is the symplectic modulation $M_w F(z) = e^{2\pi i \Omega(z,w)}F(z)$ extended dually to $M^{\infty}(\mathbb{R}^{2d})$.
\end{proposition}
It follows from \cref{weylspreadidentity} that we equivalently have
\begin{align}\label{fourwigper}
    \mathcal{F}_W (\beta_w(S)) = \mathcal{F}_W (S)(z-w).
\end{align}
As a modulation for operators, $\beta_w$ shifts naturally arise when considering the twisted convolution, which is the Fourier-Wigner transform of the composition of two operators:
\begin{align*}
    \mathcal{F}_{W} (ST)(z) = \int_{\mathbb{R}^{2d}} \mathcal{F}_W (S)(z')\mathcal{F}_W\big(\widecheck{\beta_z(T)}\big)(z')e^{-\pi i \Omega(z,z')}\, \d z'
\end{align*}
Recall that a $\Lambda$-translation invariant operator $T\in\mathcal{M}^{\infty}$ satisfies 
\begin{align}\label{transinvop}
    T = \pi(\lambda) T\pi(\lambda)^*.
\end{align}
We introduce the concept of a $\Lambda$-\emph{modulation invariant operator}:
\begin{definition}
    An operator $T\in\mathcal{M}^{\infty}$ is called $\Lambda$-modulation-invariant if 
    \begin{align}\label{mod-inv-def}
        T = \beta_{\lambda}(T)
    \end{align}
    for every $\lambda \in \Lambda$. 
\end{definition}
The phase factor in \cref{mod-inv-def} arises due the use of time--frequency shitfs $\pi(z)$ as opposed to the more suited symmetric time--frequency shifts $\rho(z)$, in the translation picture the two coincide.

Using the identity $\pi(z)^*= e^{-2\pi i x \omega}\pi(-z)$, $\Lambda$-modulation-invariance condition is equivalent to the condition
\begin{align}
    T = \pi\Big(\frac{\lambda}{2}\Big)T\pi\Big(-\frac{\lambda}{2}\Big)^*,
\end{align}
or 
\begin{align}\label{antiintertwin}
    T\pi\Big(-\frac{\lambda}{2}\Big) = \pi\Big(\frac{\lambda}{2}\Big)T.
\end{align}
While the $\Lambda$-translation invariant operators are those operators with $\Lambda$-periodic Weyl symbols, the $\Lambda$-modulation invariant operators are precisely the operators with $\Lambda$-invariant Fourier-Wigner transform by \cref{fourwigper}. Since $\Lambda$-modulation-invariance is equivalent to periodicity of the spreading function, these operators will not have any decay in spreading function and hence will not be in any Schatten class. 

We have seen that the canonical example of a $\Lambda$-translation invariant operator was the frame operator, which we can roughly think of as similar to the identity. The canonical $\Lambda$-modulation invariant operator, on the other hand, can be understood as ``reflecting'' the time-frequency concentration of a function in phase space. If we consider modulation invariance for the whole double phase space, as presented in \cite{Gr76}, we find the parity operator:
\begin{proposition}
    The parity operator $P$ is an $\mathbb{R}^{2d}$-modulation invariant operator.
\end{proposition}
\begin{proof}
    The result follows from the property $P\pi(-z) = \pi(z)P$ for every $z$ (\cref{paritytfshift}), along with the characterisation \cref{antiintertwin}.
    
\end{proof}
Discretising to a lattice gives another example similar to the discretisation of the identity to the frame operator in the case of $\Lambda$-translation invariant operators:
\begin{proposition}
    Given atoms $g,h\in L^2(\mathbb{R}^d)$, the operator
    \begin{align*}
        T = \sum_{\lambda\in\Lambda}e^{-\pi i \lambda_1\lambda_2/2}\pi\Big(\frac{\lambda}{2}\Big)g\otimes \pi\Big(\frac{\lambda}{2}\Big)^*h = \sum_{\lambda\in \Lambda}\beta_{\lambda}(g\otimes h).
    \end{align*}
    is a $\Lambda$-modulation invariant operator.
\end{proposition}
\begin{proof}
    Let $\mu \in \Lambda$. Then with $T$ as above, and for convenience denoting $g\otimes h =: S$;
    \begin{align*}
        \beta_{\mu}(T) &= \sum_{\lambda\in\Lambda}e^{-\pi i( \mu_1\mu_2 + \lambda_1\lambda_2)/2}\pi\Big(\frac{\mu}{2}\Big)\pi\Big(\frac{\lambda}{2}\Big)S\pi\Big(\frac{\lambda}{2}\Big)\pi\Big(\frac{\mu}{2}\Big) \\
        &= \sum_{\lambda\in\Lambda}e^{-\pi i( \mu_1\mu_2 + \lambda_1\lambda_2 + \mu_1\lambda_2 + \lambda_1\mu_2)/2}\pi\Big(\frac{\mu+\lambda}{2}\Big)S\pi\Big(\frac{\lambda+\mu}{2}\Big) \\
        &= \sum_{\lambda\in\Lambda}e^{-\pi i( \mu_1 + \lambda_1)(\mu_2+\lambda_2)/2}\pi\Big(\frac{\mu+\lambda}{2}\Big)S\pi\Big(\frac{\lambda+\mu}{2}\Big) \\
        &= T.
    \end{align*}
\end{proof}
We have seen that heuristically, $\Lambda$-translation invariant operators act similar to the identity, while $\Lambda$-modulation invariant operators act similar to a reflection in time-frequency concentration in phase space. As the composition of two reflection operators gives the identity, we find that the composition of two $\Lambda$-modulation invariant operators gives a $\Lambda/2$-translation invariant operator:
\begin{theorem}\label{compositionident}
    Let $S,T\in\mathcal{M}^\infty$ be two $\Lambda$-modulation invariant operators such that $ST\in\mathcal{M}^\infty$ is well defined. Then the composition $ST$ is a $\Lambda/2$-translation invariant operator.    
\end{theorem}
\begin{proof}
    Let $S,T\in\mathcal{M}^{\infty}$ be $\Lambda$-modulation invariant operators. Then for $\lambda\in\Lambda$;
    \begin{align*}
        \pi\Big(\frac{\lambda}{2}\Big)ST\pi\Big(\frac{\lambda}{2}\Big)^* &= \pi\Big(\frac{\lambda}{2}\Big)S\pi\Big(\frac{\lambda}{2}\Big)\pi\Big(\frac{\lambda}{2}\Big)^*T\pi\Big(\frac{\lambda}{2}\Big)^* \\
        &= e^{-\pi i \lambda_1\lambda_2}\pi\Big(\frac{\lambda}{2}\Big)S\pi\Big(\frac{\lambda}{2}\Big)\pi\Big(-\frac{\lambda}{2}\Big)T\pi\Big(-\frac{\lambda}{2}\Big) \\
        &= ST
    \end{align*}
    where we simply used the unitarity of $\pi(z)$ and the definition of $\Lambda$-modulation invariant operators. Hence $ST$ is $\Lambda/2$-translation invariant.
    
\end{proof}
Since $\Lambda \triangleleft \Lambda/2$, the composition of $\Lambda$-modulation invariant operators is also a $\Lambda$-translation invariant operator.
It follows then that the parity operator is in fact the only $\mathbb{R}^{2d}$-modulation invariant operator (up to a constant):
\begin{proposition}
    If $S\in\mathcal{M}^{\infty}$ is an $\mathbb{R}^{2d}$-modulation invariant operator, then $S=c\cdot P$ for some constant $c\in\mathbb{C}$.
\end{proposition}
\begin{proof}
    Let $S\in\mathcal{M}^\infty$ be an $\mathbb{R}^{2d}$-modulation invariant operator. Then \cref{fourwigper} implies that $\mathcal{F}_W(S) \equiv c$ for some $c\in\mathbb{C}$. By the uniqueness of Fourier-Wigner transform, this implies $S=c\cdot P$, since $\mathcal{F}_W (P)=2^d$.
\end{proof}

\begin{remark}
    For an $\mathbb{R}^{2d}$-modulation invariant operator $S\in\mathcal{L}(L^2)$, 
    \begin{align*}
        SP\pi(z) = \pi(z)SP
    \end{align*}
    implies $S=c\cdot P$ by irreducibility of $\pi(z)$. 
\end{remark}

\section{Decompositions of \texorpdfstring{$\Lambda$}{Lambda}-Modulation Invariant Operators}
\subsection{Analysis and Synthesis of \texorpdfstring{$\Lambda$}{Lambda}-Modulation Invariant Operators}
The periodicity of the Fourier-Wigner transform of $\Lambda$-modulation invariant operators means they can be described in terms of a trigonometric decomposition:
\begin{proposition}\label{trigdecomp}
    Given an operator $T\in\mathcal{M}^{\infty}$ which is $\Lambda$-modulation invariant, the mapping $\mathcal{F}_W(T)\mapsto \{\sigma_T(\lambda^{\circ})\}_{\Lambda^\circ}$ is a Gelfand triple isomorphism;
    \begin{align*}
        \mathcal{F}_W (T) \in \big(M^1,L^2,M^{\infty}\big)(\mathbb{R}^{2d}/\Lambda) \iff \{\sigma_T(\lambda^{\circ})\}_{\lambda^{\circ}\in \Lambda^\circ} \in \big( \ell^1,\ell^2,\ell^{\infty} \big)(\Lambda^\circ).
    \end{align*}
    Furthermore, the Fourier-Wigner transform $\mathcal{F}_W(T)$ can be trigonometrically decomposed as
    \begin{align*}
        \mathcal{F}_W (T)(z) = \sum_{\lambda^{\circ}\in\Lambda^{\circ}} \sigma_T(\lambda^{\circ}) e^{2\pi i \Omega(\lambda^{\circ},z)}.
    \end{align*}
    
\end{proposition}
\begin{proof}
    By the definition of $\Lambda$-modulation-invariant along with \cref{fourwigper}, the Fourier-Wigner transform of some $\Lambda$-modulation invariant $T\in\mathcal{M}^{\infty}$ is a $\Lambda$-periodic function. By \cref{classical-sfs}, the trigonometric decomposition follows. To show the Gelfand triple isomorphism, the space $M^1(\mathbb{R}^{2d}/\Lambda)$ is equal to $\mathcal{A}(\mathbb{R}^{2d}/\Lambda)$ since $\mathbb{R}^{2d}/\Lambda$ is compact (cf. Lemma 4.1, \cite{Ja18}), and hence the Fourier series is absolutely convergent, and conversely an absolutely convergent Fourier series implies $\mathcal{F}_W(T)\in\mathcal{A}(\mathbb{R}^{2d}/\Lambda)$. The $M^1(\mathbb{R}^{2d}/\Lambda)$ and $M^\infty(\mathbb{R}^{2d}/\Lambda)$ results then follow. The $L^2(\mathbb{R}^{2d}/\Lambda)$ case is Parseval's identity.
    
\end{proof}
By \cref{weylspreadidentity}, we can formulate \cref{trigdecomp} as a statement on the support of the Weyl symbol:
\begin{corollary}\label{discreteweyl}
    Given an operator $T\in\mathcal{M}^{\infty}$ which is $\Lambda$-modulation-invariant, the Weyl symbol $\sigma_T$ is given by
    \begin{align*}
        \sigma_T(z) = \sum_{\lambda^{\circ}\in\Lambda^{\circ}} \sigma_T(\lambda^{\circ})\delta_{\lambda^{\circ}}(z),
    \end{align*}
    where the sum converges in the weak-$*$ topology of $M^{\infty}(\mathbb{R}^d)$, i.e. the Weyl symbol is supported on the adjoint lattice.
\end{corollary}
We have already seen that operators of the type 
\begin{align*}
    T = \sum_{\lambda\in\Lambda}e^{-\pi i \lambda_1\lambda_2/2}\pi\Big(\frac{\lambda}{2}\Big)g\otimes \pi\Big(\frac{\lambda}{2}\Big)^*g
\end{align*}
are $\Lambda$-modulation invariant. However on the converse, the surjectivity of the periodisation operator onto $\mathcal{A}(\mathbb{R}^{2d}/\Lambda)$ informs us that in fact any ``smooth'' $\Lambda$-modulation invariant operator is of this form, for some operator (not necessarily rank-one) in $\mathcal{M}^1$:
\begin{proposition}\label{spreadingperiodisation}
    Given an operator $T\in\mathcal{M}^{\infty}$ which is $\Lambda$-modulation invariant, such that $\mathcal{F}_W (T)\in \mathcal{A}(\mathbb{R}^{2d}/\Lambda)$, $T$ can be expressed as
    \begin{align*}
        T = \sum_{\lambda\in\Lambda}e^{-\pi i \lambda_1\lambda_2/2}\pi\Big(\frac{\lambda}{2}\Big) S \pi\Big(\frac{\lambda}{2}\Big)
    \end{align*}
    for some $S\in\mathcal{M}^1$.
\end{proposition}
\begin{proof}
    Since $T$ is $\Lambda$-modulation-invariant and $\mathcal{F}_W (T)\in \mathcal{A}(\mathbb{R}^{2d}/\Lambda)$, the surjectivity of the periodisation operator onto $\mathcal{A}(\mathbb{R}^{2d}/\Lambda)$ means there must exist some $g\in M^1(\mathbb{R}^{2d})$ such that $\sum_{\lambda\in\Lambda} T_{\lambda} g = \mathcal{F}_W (T)$. Defining $\mathcal{F}_W (S) = g$ then gives an appropriate operator.
\end{proof}
Motivated by our previous results and operator-periodisation, we introduce the following.
\begin{definition}
    For $T\in \mathcal{M}^{\infty}$, the \emph{Fourier-Wigner periodisation} of $T$ with respect to $\Lambda$ is given by $\sum_{\lambda\in \Lambda}\beta_{\lambda}(T).$
\end{definition}
\noindent The definition above is motivated by the fact that the Weyl correspondence sends $\sum_{\lambda\in \Lambda}\beta_{\lambda}(T)$ to $\sum_{\lambda\in \Lambda}M_{\lambda}\sigma_{T}$, however \cref{weylspreadidentity} shows that $$\mathcal{F}_{\Omega}\left(\sum_{\lambda\in \Lambda}M_{\lambda}\sigma_T\right) = \sum_{\lambda\in \Lambda}T_{\lambda}\mathcal{F}_{\Omega}\sigma_T = \sum_{\lambda\in \Lambda}T_{\lambda}\mathcal{F}_{W}(T), \qquad \forall z\in \mathbb{R}^{2d}.$$
Therefore, $\sum_{\lambda\in \Lambda}M_{\lambda}\sigma_T$ is exactly the symplectic Fourier-transform of the periodisation of the Fourier-Wigner transform of $T$, which corresponds to the operator $\sum_{\lambda\in \Lambda}\beta_{\lambda}(T).$ 

\subsection{Sampled Convolutions}
Recalling the definition of the convolution of two operators:
\begin{align*}
    T \star S(z) := \mathrm{tr}(T\alpha_z(\Check{S})),
\end{align*}
it is an interesting question in which cases one can reconstruct an operator $T$ from samples of $T\star S$ on, for example, a lattice. This problem appears in several settings: In the case of two rank--one operators $T=f\otimes f$, $S=g\otimes g$ for $f,g\in L^2(\mathbb{R}^d)$, the operator convolution gives the spectrogram;
\begin{align*}
    (T \star S)(z) = |V_g f(z)|^2,
\end{align*}
while for a rank-one $S=g\otimes g$ but arbitrary $T$, the discretisation of $T\star S$ corresponds to the diagonal of the Gabor matrix
\begin{align*}
    T \star S(\lambda) = \langle T\pi(\lambda)g, \pi(\lambda)g \rangle_{\mathscr{S}^\prime,\mathscr{S}}.
\end{align*}
In the former case it is known (cf. \cite{GrLi22}) that there exists no $g$ such that $T \mapsto S\star T|_{\Lambda}$ is injective for the space of Hilbert-Schmidt operators. On the other hand in the latter case, an operator $T\in \mathfrak{S}^\prime$ with a symbol in some Paley-Wiener space \textit{can} be reconstructed from a discretisation of $S \star T$, for an appropriately chosen $g$ and $\Lambda$ \cite{GrPa14}. 

This line of inquiry motivates $\Lambda$-modulation invariant operators, since such operators can be reconstructed from samples of convolution with an appropriate $S$. In what follows we let $K$ denote a compact neighborhood of the origin not containing any $\lambda^\circ \in \Lambda^\circ$ apart from the origin.
\begin{theorem}
    Let $S\in\mathcal{M}^1$ such that $\sigma_S(0)\neq 0$, and  $\mathrm{supp}(\sigma_S)\subset K$. Then given any $\Lambda$-modulation invariant operator $T\in\mathcal{M}^{\infty}$; 
    \begin{align*}
        T = \frac{1}{\sigma_S(0)}\sum_{\lambda^\circ\in\Lambda^\circ} T\star S(\lambda^\circ)\cdot e^{-4i\pi\lambda^\circ_1\lambda^\circ_2} \pi(2\lambda^\circ)P 
    \end{align*}
\end{theorem}
\begin{proof}
    We will use the fact that the convolution of two operators is equal to the convolution of their Weyl symbols. By \cref{discreteweyl}, the Weyl symbol of $T$ is supported on the lattice $\Lambda^\circ$, and can be expressed as 
    \begin{align*}
        \sigma_T(z) = \sum_{\lambda^\circ\in\Lambda^\circ} \sigma_T(\lambda^\circ)\delta_{\lambda^\circ}(z).
    \end{align*}
    Given our conditions on $S$ and $K$ then follows that 
    \begin{align*}
        (T \star S)(\lambda^\circ) &= (\sigma_T * \sigma_S) (\lambda^\circ)\\
        &= \sigma_T (\lambda^\circ)\cdot \sigma_S (0),
    \end{align*}
    and so we can rewrite
    \begin{align*}
        \sigma_T(z) = \frac{1}{\sigma_S(0)}\sum_{\lambda^\circ\in\Lambda^\circ} T \star S(\lambda^\circ) \delta_{\lambda^\circ}(z).
    \end{align*}
    By the second statement of \cref{weylquant}, the Weyl quantisation of $\delta_{z}$ is the operator $e^{-4i\pi x\cdot\omega}\pi(2z)P$, and so 
    \begin{align*}
        T = {\sigma_S(0)}\sum_{\lambda^\circ\in\Lambda^\circ} T \star S(\lambda^\circ) e^{-4i\pi x\cdot\omega}\pi(2z)P.
    \end{align*}
\end{proof}

\subsection{A Janssen Type Representation and Boundedness of \texorpdfstring{$\Lambda$}{Lambda}-Modulation Invariant Operators}
For the canonical example of $\Lambda$-translation invariant operators, the frame operator $S_{g,\Lambda}$, the Janssen's representation can be interpreted as Poisson's summation formula applied to the periodisation of the Weyl symbol of $g\otimes g$. Since the spreading function of every $\Lambda$-modulation invariant operator $T$ with $\mathcal{F}_W (T)\in \mathcal{A}(\mathbb{R}^{2d}/\Lambda)$ can be written as the periodisation of the spreading function of some $S\in\mathcal{M}^1$, we can follow the same approach for $\Lambda$-modulation invariant operators, but we find a subtle difference in the lattices of reconstruction due to \cref{parityintertw}. We begin by considering the spreading quantisation of $e^{2\pi i \Omega(z,z')}$: 
\begin{lemma}\label{parityquant}
    For any $z\in\mathbb{R}^{2d}$, the Fourier-Wigner transform of $e^{-\pi i x\cdot\omega}\pi(z)P \in \mathcal{M}^{\infty}$ is the distribution $f(z')=2^{-d} e^{\pi i \Omega(z,z')}$ in $M^{\infty}(\mathbb{R}^{2d})$.
\end{lemma}
\begin{proof}
    From the second form in \cref{weylquant}, the Weyl symbol of $e^{-\pi i x\cdot\omega}\pi(z)P$ is the Dirac delta distribution $\delta_{2z}$. Using the correspondence of Weyl symbol and Fourier-Wigner via the symplectic Fourier transform, the Fourier-Wigner of transform $e^{-\pi i x\cdot\omega}\pi(z)P$ is thus given by
    \begin{align*}
        \mathcal{F}_W(e^{-\pi i x\cdot\omega}\pi(z)P) &= \mathcal{F}_{\Omega} (\delta_{2z})(z') \\
        &= 2^{-d} e^{\pi i \Omega(z,z')}.
    \end{align*}
\end{proof}
Note that for some $\lambda\in\Lambda$, the Fourier-Wigner transform of $e^{-\pi i \lambda_1\cdot \lambda_2}\pi(\lambda)P$ is $\Lambda/2$-periodic as opposed to merely $\Lambda$-periodic, as a result of \cref{parityintertw}, and we see the result of this in the Janssen type representation of $\Lambda$-modulation invariant operators:  
\begin{theorem}\label{janssenrep}
    Let $S\in\mathcal{M}^1$. Then 
    \begin{align*}
        \sum_{\lambda\in\Lambda} \beta_{\lambda} (S) = \frac{2^d}{|\Lambda|} \sum_{\lambda^{\circ}\in\Lambda^{\circ}} \sigma_S(\lambda^{\circ})\cdot e^{-4\pi i \lambda^{\circ}_1\cdot\lambda^{\circ}_2 }  \pi(2\lambda^{\circ})P,
    \end{align*}
    or equivalently
    \begin{align*}
        \sum_{\lambda\in\Lambda} \beta_{\lambda} (S) = \frac{2^d}{|\Lambda|} \sum_{\lambda^{\circ}\in\Lambda^{\circ}} \sigma_S(\lambda^{\circ})\cdot \alpha_{\lambda^{\circ}}(P).
    \end{align*}
\end{theorem}
\begin{proof}
     We consider the Fourier-Wigner transform of the left hand side sum;
     \begin{align*}
         \mathcal{F}_W \Big( \sum_{\lambda\in\Lambda} \beta_{\lambda} (S) \Big)(z) &= \sum_{\lambda\in\Lambda} \mathcal{F}_W \big(\beta_{\lambda} (S)\big)(z) \\
         &= \sum_{\lambda\in\Lambda} T_{\lambda}\big(\mathcal{F}_W (S)\big)(z).
     \end{align*}
     By \cref{symplpoisson}, the sum converges absolutely for every $z\in\mathbb{R}^{2d}$, for $S\in\mathcal{M}^1$. The symplectic Poisson summation formula then gives 
     \begin{align*}
         \sum_{\lambda\in\Lambda} T_{\lambda}\big(\mathcal{F}_W (S)\big)(z) &= \frac{1}{|\Lambda|} \sum_{\lambda^{\circ}\in\Lambda^{\circ}} \mathcal{F}_{\Omega} (\mathcal{F}_W (S))(\lambda^\circ)e^{2\pi i \Omega(\lambda^\circ,z)} \\
         &= \frac{1}{|\Lambda|} \sum_{\lambda^{\circ}\in\Lambda^{\circ}} \sigma_S(\lambda^\circ)e^{2\pi i \Omega(\lambda^\circ,z)}.
     \end{align*}
     The result the follows from \cref{parityquant}, and the equivalent form from \cref{parityintertw}.
     
\end{proof}
We can also show the boundedness of $\Lambda$-modulation-invariant operators on shift-invariant, and in particular modulation spaces:
\begin{theorem}
    Let $T\in\mathcal{M}^\infty$ be a $\Lambda$-modulation-invariant operator, such that $\mathcal{F}_W(T) \in \mathcal{A}(\mathbb{R}^{2d}/\Lambda)$. Then $T$ is bounded on all modulation spaces $M^{p,q}(\mathbb{R}^{d})$, with
    \begin{align*}
        \|T\|_{\mathcal{L}(M^{p,q})} \leq C_{\Lambda}\| \mathcal{F}_W(T) \|_{\mathcal{A}},
    \end{align*}
    where $C_{\Lambda} = \frac{2^d}{|\Lambda|}$.
\end{theorem}
\begin{proof}
    Since $\mathcal{F}_W(T) \in \mathcal{A}(\mathbb{R}^{2d}/\Lambda)$, by \cref{trigdecomp} and \cref{parityquant} we can write $T$ as 
    \begin{align*}
        T = C_{\Lambda}\sum_{\lambda^\circ\in\Lambda^\circ} \sigma_T(\lambda^\circ)\cdot e^{-4\pi i \lambda^\circ_1\cdot\lambda^\circ_2}\pi(2\lambda^\circ)P,
    \end{align*}
    where $\{c_{\lambda^\circ}\}_{\Lambda^\circ} \in \ell^1(\Lambda^\circ)$. Since the modulation spaces $M^{p,q}(\mathbb{R}^{d})$ are shift invariant, for any $f\in M^{p,q}(\mathbb{R}^{d})$;
    \begin{align*}
        \|Tf\|_{M^{p,q}} &= C_{\Lambda}\Big\|\sum_{\lambda^\circ\in\Lambda^\circ} \sigma_T(\lambda^\circ)\cdot e^{-4\pi i \lambda^\circ_1\cdot\lambda^\circ_2}\pi(2\lambda^\circ)Pf\Big\|_{M^{p,q}} \\
        &\leq C_{\Lambda}\sum_{\lambda^\circ\in\Lambda^\circ} |\sigma_T(\lambda^\circ)| \|\pi(2\lambda^\circ)Pf\|_{M^{p,q}} \\
        &= C_{\Lambda}\sum_{\lambda^\circ\in\Lambda^\circ} |\sigma_T(\lambda^\circ)| \|f\|_{M^{p,q}} \\
        &= C_{\Lambda} \|T\|_{\mathcal{A}}\|f\|_{M^{p,q}}.
    \end{align*}
\end{proof}

\section{A Calculus for \texorpdfstring{$\Lambda$}{Lambda}-Modulation Invariant Operators}
We now consider the composition of $\Lambda$-modulation-invariant operators. We reiterate that these will no longer be $\Lambda$-modulation invariant, but rather $\Lambda/2$-translation invariant. We also restrict our attention to operators in $\mathcal{L}(L^2(\mathbb{R}^d))$, so that their composition is well defined. We find the following calculus.
\begin{theorem}\label{spreadfunccalc}
    Let $S,T\in\mathcal{L}(L^2(\mathbb{R}^d))$ be two $\Lambda$-modulation invariant operators, and let $2\Omega(\Lambda^\circ, \Lambda^\circ) \subset \mathbb{Z}$, that is $2\Omega(\lambda^\circ, \mu^\circ)=n\in\mathbb{Z}$ for every $\lambda^\circ,\mu^\circ\in\Lambda^\circ$. Then $S$ and $T$ satisfy the spreading function calculus
    \begin{align*}
        \mathcal{F}_W(ST)(2z) = 2^{2d}\sigma_S * \Check{\sigma}_T(z).
    \end{align*}
\end{theorem}
\begin{proof}
    Let $S,T\in\mathcal{L}(L^2(\mathbb{R}^d))$ be $\Lambda$-modulation invariant operators where $2\Omega(\Lambda^\circ, \Lambda^\circ) \subset \mathbb{Z}$. Then by \cref{discreteweyl} we can consider the discrete representations
    \begin{align*}
        \sigma_S &= \sum_{\lambda^\circ\in\Lambda^\circ} c_{\lambda^\circ}\delta_{\lambda^\circ} \\
        \sigma_T &= \sum_{\mu^\circ\in\Lambda^\circ} d_{\mu^\circ}\delta_{\mu^\circ}.
    \end{align*}
    We then use the representation given by \cref{weylquant}, giving
    \begin{align*}
        S &= 2^d\sum_{\lambda^\circ\in\Lambda^\circ} c_{\lambda^\circ} e^{-4\pi i \lambda_1^\circ\lambda_2^\circ}\pi(2\lambda^\circ)P \\
        T &= 2^d\sum_{\mu^\circ\in\Lambda^\circ} d_{\mu^\circ}e^{-4\pi i \mu_1^\circ\mu_2^\circ}\pi(2\mu^\circ)P.
    \end{align*}
    We then calculate the composition using the twisted convolution to find
    \begin{align*}
        ST &= 2^{2d}\sum_{\lambda^\circ,\mu^\circ} c_{\lambda^\circ}d_{\mu^\circ} e^{-4 \pi i (\lambda_1^\circ\lambda_2^\circ+\mu_1^\circ\mu_2^\circ)}\pi(2\lambda^\circ)P\pi(2\mu^\circ)P \\
        &= 2^{2d}\sum_{\lambda^\circ,\mu^\circ} c_{\lambda^\circ}d_{\mu^\circ} e^{-4 \pi i (\lambda_1^\circ\lambda_2^\circ+\mu_1^\circ\mu_2^\circ)}\pi(2\lambda^\circ)\pi(-2\mu^\circ) \\
        &= 2^{2d}\sum_{\lambda^\circ,\mu^\circ} c_{\lambda^\circ}d_{\mu^\circ} e^{-4 \pi i (\lambda_1^\circ\lambda_2^\circ+\mu_1^\circ\mu_2^\circ)}e^{8\pi i \lambda_1^\circ \mu_2^\circ}\pi(2\lambda^\circ-2\mu^\circ).
    \end{align*}
    Note that
    \begin{align*}
        -4 \pi i (\lambda_1^\circ\lambda_2^\circ+\mu_1^\circ\mu_2^\circ) + 8\pi i \lambda_1^\circ \mu_2^\circ = -4 \pi i (\lambda_1^\circ-\mu_1^\circ)(\lambda_2^\circ-\mu_2^\circ) - 4\pi i\Omega(\lambda^{\circ},\mu^{\circ}),
    \end{align*}
    so by our assumption $2\Omega(\Lambda^\circ, \Lambda^\circ) \subset \mathbb{Z}$, the exponential $e^{- 4\pi i\Omega(\lambda^\circ,\mu^\circ)}$ of the last twisting term will be $1$. Hence, using the change of variables $\gamma^\circ := \lambda^\circ-\mu^\circ$ we find
    \begin{align*}
        ST &= 2^{2d}\sum_{\lambda^\circ,\mu^\circ} c_{\lambda^\circ}d_{\mu^\circ} e^{-4 \pi i (\lambda_1-\mu_1)(\lambda_2-\mu_2)}\pi(2\lambda^\circ-2\mu^\circ) \\
        &= 2^{2d}\sum_{\lambda^\circ,\gamma^\circ} c_{\lambda^\circ}d_{\lambda^\circ - \gamma^\circ} e^{-4 \pi i \gamma_1\gamma_2}\pi(2\gamma^\circ)
    \end{align*}
    Then since the spreading function of a time-frequency shift $\pi(z)$ is given by $\delta_z$, by the relationship between the spreading function and Fourier-Wigner transform given by \cref{spread-weyl-relations}, we conclude that
    \begin{align*}
        \mathcal{F}_W(ST)(2\gamma^\circ) &= 2^{2d}\sum_{\lambda^\circ,\gamma^\circ} c_{\lambda^\circ}d_{\lambda^\circ - \gamma^\circ} \\
        &= 2^{2d} \sigma_S * \check{\sigma}_T(\gamma^\circ),
    \end{align*}
    and the result follows.
    
\end{proof}

\begin{remark}\label{im-T-condition}\normalfont
    The above calculus holds in general for $S,T\in\mathcal{M}^\infty$ where $\operatorname{Im}(T)\subset M^1(\mathbb{R}^d)$
\end{remark}
We collect some simple corollaries of the spreading function calculus. Firstly, since a bounded operator is uniquely determined by its spreading function, we have the following composition property of $\Lambda$-modulation invariant operators:
\begin{corollary}\label{commcalc}
    If $S,T$ and $\Lambda$ are as in \cref{spreadfunccalc}, then $S\Check{T}=T\Check{S}$.
\end{corollary}
Secondly, using \cref{weylspreadidentity} along with the fact the symplectic Fourier transform is its own inverse, we can reformulate \cref{spreadfunccalc} in terms of the Weyl symbol:
\begin{corollary}\label{weylsymbcalc}
    For $S,T$ and $\Lambda$ as in \cref{spreadfunccalc};
    \begin{align*}
        \sigma_{ST}(\tfrac{z}{2}) = 2^{2d}\left(\mathcal{F}_W (S)\overline{\mathcal{F}_W (T)}\right)(z).
    \end{align*}
\end{corollary}
These characterisations of $\Lambda$-modulation invariant operators gives a way to identify each $\Lambda$-modulation invariant operator with a $\Lambda/2$-translation invariant operator by composing with the parity operator, and vice versa:
\begin{theorem}\label{correspondence}
    The map $\Theta: S\mapsto SP$ is a self-inverse isometric isomorphism in $\mathcal{M}^{\infty}$-norm. It identifies $\Lambda$-modulation invariant operators with $\Lambda/2$-translation invariant operators. Moreover, in terms of symbols:
    \begin{align*}
        \mathcal{F}_W (\Theta(S))(2z) = 2^{d}\sigma_S(z).
    \end{align*}
\end{theorem}
\begin{proof}
    We will only show that $\Theta$ is an isometry on and $\mathcal{M}^{\infty}$, as this will be enough to show that $\Theta$ is a self-inverse isometric isomorphism on these spaces following Lemma \ref{parity-isometry}. Let $S\in \mathcal{M}^{\infty}$, then it follows from the fact that $P_{|M^1(\mathbb{R}^d)}:M^1(\mathbb{R}^d)\to M^1(\mathbb{R}^d)$ is an isometric isomorphism via Lemma \ref{parity-isometry} that we have the following:
\begin{align*}
    \|SP\|_{\mathcal{M}^{\infty}} &= \sup\left\{\left|\rin{M^{\infty},M^1}{SPf}{g}\right|:\|f\|_{M^1}=\|g\|_{M^1}=1  \right\} \\
    &= \sup\left\{\left|\rin{M^{\infty},M^1}{Sf}{g}\right|:  
\|f\|_{M^1}=\|g\|_{M^1}=1 \right\} \\
&=\|S\|_{\mathcal{M}^{\infty}}.
\end{align*}
It follows that $\Theta$ is an isometric isomorphism on $\mathcal{M}^{\infty}.$

Let $S$ be a $\Lambda$-modulation invariant operator. We recall that the parity operator $P$ is an $\mathbb{R}^{2d}$-modulation invariant operator, and consequently a $\Lambda$-modulation invariant operator for any $\Lambda$. Hence by \cref{compositionident}, $SP$ is a $\Lambda/2$-translation invariant operator. Conversely, given some $\Lambda/2$-translation invariant operator $T$, the operator $TP$ is a $\Lambda$-modulation invariant operator, since
\begin{align*}
    TP\pi(\tfrac{\lambda}{2}) &= T\pi(-\tfrac{\lambda}{2})P \\
    &= \pi(-\tfrac{\lambda}{2})TP,
\end{align*}
and so the result follows from \cref{antiintertwin}. Moreover since the Weyl symbol of $P$ is the Dirac delta distribution, by the same argument as the proof of \cref{spreadfunccalc} and \cref{im-T-condition}, $2^d \sigma_{S}(z) = \mathcal{F}_W (SP)(2z)$.

\end{proof}

\section{\texorpdfstring{$\Lambda$}{Lambda}-Modulation Invariant Operators via Finite-Rank Operators Generated by the Heisenberg Module}\label{lattice-discussion}
We revisit the characterisation results of Section \ref{lambda-trans-adj} now equipped with the correspondence result of Theorem \ref{correspondence}. Just like how we consider the Gabor frame operator $S_{g,h,\Lambda}$ the canonical $\Lambda$-translation invariant operator equivalent to the operator-periodisation $\sum_{\lambda\in \Lambda}\alpha_{\lambda}(g\otimes h)$, we analogously find that $\Lambda/2$-periodisation of $g\otimes h$ pre-composed with the parity operator gives us:
\begin{align}\label{motiv-mod}
    \sum_{\lambda\in \Lambda}\alpha_{\lambda/2}(g\otimes h)P = \sum_{\lambda\in \Lambda}\beta_{\lambda}(Pg\otimes h)=S_{g,h,\Lambda/2}P.
\end{align}
Since $P$ is unitary in $\mathcal{L}\left(L^2(\mathbb{R}^d)\right)$, \cref{motiv-mod} motivates us to consider operators of the form
\begin{align}\label{canonical-mod}
    M = \sum_{\lambda\in \Lambda}\beta_{\lambda}(g\otimes h)
\end{align}
as the canonical form of $\Lambda$-modulation invariant operators. Indeed we shall show, just like in the case of Section \ref{lambda-trans-adj}, that $\Lambda$-modulation invariant operators can be characterised in terms of Fourier-Wigner periodisation of finite-rank operators with generators coming from the Heisenberg module.
\begin{lemma}\label{parity-and-heis}
    Fix any lattice $\Lambda\subseteq \mathbb{R}^{2d}$. If $\psi_1,\psi_2\in L^2(\mathbb{R}^d)$, then
    \begin{align*}
        S_{P\psi_1,P\psi_2,\Lambda} = PS_{\psi_1,\psi_2,\Lambda}P=\widecheck{S_{\psi_1,\psi_2,\Lambda}}.
    \end{align*}
    As a corollary, the parity operator restricts to an isometric isomorphism $P_{|\mathcal{E}_{\Lambda}(\mathbb{R}^d)}:\mathcal{E}_{\Lambda}(\mathbb{R}^d)\to \mathcal{E}_{\Lambda}(\mathbb{R}^d)$ on the Heisenberg module.
\end{lemma}
\begin{proof}
    We compute for any $g\in L^2(\mathbb{R}^d):$
    \begin{align*}
        S_{P\psi_1,P\psi_2,\Lambda}g &= \sum_{\lambda\in \Lambda}\rin{L^2}{g}{\pi(\lambda)P\psi_1}\pi(\lambda)P\psi_2\\
        &= P\left(\sum_{\lambda\in \Lambda}\rin{L^2}{Pg}{\pi(-\lambda)\psi_1}\pi(-\lambda)\psi_2 \right)\\
        &= PS_{\psi_1,\psi_2,\Lambda}Pg =\widecheck{S_{\psi,\Lambda}}g.
    \end{align*}
    Since $P$ is a unitary map in $\mathcal{L}\left(L^2(\mathbb{R}^d)\right)$, it follows in particular that for any $\psi\in L^2(\mathbb{R}^d)$:
    \begin{equation}\label{frame-parity}
        \begin{split}
            \|S_{P\psi,\Lambda}\|_{\mathcal{L}(L^2)}^2&= \sup_{\|g\|_{L^2}=1}\rin{L^2}{PS_{\psi,\Lambda}Pg}{PS_{\psi,\Lambda}Pg}\\
            &= \sup_{\|g\|_{L^2}=1} \rin{L^2}{S_{\psi,\Lambda}Pg}{S_{\psi,\Lambda}Pg}\\
            &=\sup_{\|g\|_{L^2}=1}\rin{L^2}{S_{\psi,\Lambda}g}{S_{\psi,\Lambda}g} = \|S_{\psi,\Lambda}\|_{\mathcal{L}(L^2)}^2.
        \end{split}
    \end{equation}
    It then follows from \cref{frame-parity} and \cref{completion-for-heis} that if $f\in M^1(\mathbb{R}^d)$, then:
    \begin{align}\label{parity-heis-isometry}
        \|Pf\|_{\mathcal{E}_{\Lambda}(\mathbb{R}^d)}^2=\|S_{Pf,\Lambda}\|_{\mathcal{L}(L^2)} = \|S_{f,\Lambda}\| = \|f\|_{\mathcal{E}_{\Lambda}(\mathbb{R}^d)}^2.
    \end{align}
    So $P:M^1(\mathbb{R}^d)\to M^1(\mathbb{R}^d)$ is an isometric isomorphism with respect to the $\mathcal{E}_{\Lambda}(\mathbb{R}^d)$-norm. Therefore, $P$ must extend to an isometric isomorphism $P:\mathcal{E}_{\Lambda}(\mathbb{R}^d)\to \mathcal{E}_{\Lambda}(\mathbb{R}^d)$. 
\end{proof}
\begin{theorem}\label{lamb-mod-operator-limit}
    There exists an $n_0\in \mathbb{N}$ such that every $\Lambda$-modulation invariant operator $T\in\mathcal{M}^{\infty}$ is the weak-$*$ limit of Fourier-Wigner operator-periodisations of rank-$n_0$ operators generated by functions coming from $\mathcal{E}_{\Lambda/2}(\mathbb{R}^d).$ In particular, if the samples $\{\sigma_T(\lambda^{\circ})\}_{\lambda^{\circ}\in \Lambda^{\circ}}$ 
 are in $\ell^p(\Lambda^{\circ})$ with $p\geq 2$ (resp. $1\leq p < 2$), then $T$ is the $\mathcal{M}^{\infty}$-norm limit of Fourier-Wigner operator-periodisations of rank-$n_0$ operators generated by functions coming from $\mathcal{E}_{\Lambda/2}(\mathbb{R}^d)$ (resp. $M^1(\mathbb{R}^d)$).
\end{theorem}
\begin{proof}
    Fix $n_0\in \mathbb{N}$ as in Theorem \ref{lamb-trans-operator-limit} applied to $\Lambda/2$-translation invariant operators. Let $T\in \mathcal{M}^{\infty}$ be a $\Lambda$-modulation invariant operator with samples $\{\sigma_T(\lambda^{\circ})\}_{\lambda^{\circ}\in \Lambda^{\circ}}\in \ell^{\infty}(\Lambda^{\circ})$. It follows from Theorem \ref{correspondence} that there exists an $\Lambda/2$-translation invariant operator $S\in \mathcal{M}^{\infty}$ with operator Fourier coefficients $\mathbf{k}\in \ell^{\infty}(\Lambda^{\circ})$ such that $T=SP$. Let $\{S_i\}_{i\in I}$ be a net of $\Lambda$-translation invariant operators with corresponding Fourier coefficients $\mathbf{k}_i\in B_1$ such that $\mathbf{k}_i\to \mathbf{k}$ in the weak-$*$ sense. For each $S_0\in \mathcal{M}^1$, let us compute $\rin{\mathcal{M}^{\infty},\mathcal{M}^1}{SP-S_iP}{S_0}$, and following \cite{Sk21}, we write an explicit form of the duality:
    \begin{equation}
        \begin{split}\label{explicit-duality}
        \rin{\mathcal{M}^{\infty},\mathcal{M}^1}{SP-S_iP}{S_0} &= \sum_{\lambda^{\circ}\in \Lambda^{\circ}}(k_{\lambda^{\circ}}-k_{i,\lambda^{\circ}})\rin{\mathcal{M}^{\infty},\mathcal{M}^1}{e^{-\pi i \lambda_1^{\circ}\cdot\lambda_2^{\circ}}P} {S_0}\\ &= \sum_{\lambda^{\circ}\in \Lambda^{\circ}}(k_{\lambda^{\circ}}-k_{i,\lambda^{\circ}})\rin{M^{\infty},M^1}{\delta_{2\lambda^{\circ}}}{\sigma_{S_0}} \\
        &=\rin{\ell^{\infty},\ell^1}{\mathbf{k}-\mathbf{k}_i}{\{\sigma_{S_0}(2\lambda^{\circ})\}_{\lambda^{\circ}\in \Lambda^{\circ}} }.
        \end{split}
    \end{equation}
    The last equation makes sense because $S_0\in \mathcal{M}^1$, so the samples $\sigma_{S_0}(2\lambda^{\circ})\}_{\lambda^{\circ}\in \Lambda^{\circ}}$ will be in $\ell^1(\Lambda^{\circ}).$ The net in  \cref{explicit-duality} goes to $0$ in the limit due to the last equation, and since $S_0$ is arbitrary, we have shown that $S_iP\to T$ in the weak-$*$ sense. It follows from Theorem \ref{lamb-trans-operator-limit}, Lemma \ref{parity-and-heis}, and \cref{motiv-mod} that each $S_iP$ is the Fourier-Wigner operator-periodisations of rank-$n_0$ operators with generators from $\mathcal{E}_{\Lambda/2}(\mathbb{R}^d).$

    The proof of the latter statement is similar to Theorem \ref{lamb-trans-operator-limit}, and can also be obtained using Theorem \ref{correspondence} along with the isometry property of $\Theta$ on $\mathcal{M}^{\infty}$-norm.
\end{proof}
Following the results for $\Lambda$-translation invariant operators, we shall also consider $\Lambda$-modulation invariant operators in $\mathcal{L}\left(L^2(\mathbb{R}^d)\right)$, which we shall denote by $$\mathcal{M}_{\Lambda}\left(L^2(\mathbb{R}^d)\right):= \left\{T\in L^2(\mathbb{R}^d): T=\beta_{\lambda}(T),\ \forall \lambda\in \Lambda  \right\}.$$ Here we have a crucial observation regarding the map $\Theta$ of Theorem \ref{correspondence}, it is well-defined as a map in $L^2(\mathbb{R}^d)$, in fact we have the following result.
\begin{lemma}\label{l2-correspondence}
    $\Theta$ is a self-inverse isometric isomorphism in $\mathcal{L}\left(L^2(\mathbb{R}^d)\right)$. As a map in $\mathcal{L}\left(L^2(\mathbb{R}^d)\right)$, it is also strong-operator and weak-operator continuous. Finally, it identifies $\mathcal{M}_{\Lambda}\left(L^2(\mathbb{R}^d)\right)$ with $\mathcal{L}_{\Lambda/2}\left(L^2(\mathbb{R}^d)\right)$.
\end{lemma}
\begin{proof}
    We shall only show that $\Theta$ is an isometry as a map in $\mathcal{L}\left(L^2(\mathbb{R}^d)\right)$. Let $S\in \mathcal{L}\left(L^2(\mathbb{R}^d)\right)$, because $P\in \mathcal{L}\left(L^2(\mathbb{R}^d)\right)$ is unitary, we obtain $$\|SP\|_{\mathcal{L}(L^2)}^2=\sup_{\|g\|_{L^2}=1}\rin{L^2}{SPg}{SPg} = \sup_{\|Pg\|_{L^2}=1}\rin{L^2}{Sg}{Sg}=\|S\|_{\mathcal{L}(L^2)}^2.$$ In light of Lemma \cref{parity-isometry}, this is enough to show that $\Theta$ is an isometric isomorphism on $\mathcal{L}\left(L^2(\mathbb{R}^d)\right).$ 
    
    Next, let $\{T_i\}_{i\in I}$ be a net in $\mathcal{L}\left(L^2(\mathbb{R}^d)\right)$ that converges to $T\in \mathcal{L}\left(L^2(\mathbb{R}^d)\right)$ in the weak-operator topology, that is: $\rin{L^2}{T_if}{g}\to \rin{L^2}{Tf}{g}$ for all $f,g\in L^2(\mathbb{R}^d)$. It follows from unitarity of $P\in \mathcal{L}\left(L^2(\mathbb{R}^d)\right)$ that $\rin{L^2}{T_iPf}{g}\to \rin{L^2}{TPf}{g}$ for all $f,g\in L^2(\mathbb{R}^d)$, equivalently $\Theta(T_i)\to \Theta(T)$ in the weak-operator topology. The proof for continuity of $\Theta:L^2(\mathbb{R}^d)\to L^2(\mathbb{R}^d)$ in strong-operator topology is similar. 

    The proof showing the identification of $\mathcal{M}_{\Lambda}\left(L^2(\mathbb{R}^d)\right)$ with $\mathcal{L}_{\Lambda/2}\left(L^2(\mathbb{R}^d)\right)$ via $\Theta$ is the same as in Theorem \ref{correspondence}.
\end{proof}
\begin{theorem}\label{final-mod-vn}
    There exists an $n_0\in \mathbb{N}$ such that every operator in $\mathcal{M}_{\Lambda}\left(L^2(\mathbb{R}^d)\right)$ is a weak-operator, or strong-operator limit of Fourier-Wigner operator-periodisations of rank-$n_0$ operators generated by functions coming from $\mathcal{E}_{\Lambda/2}(\mathbb{R}^d).$
\end{theorem}
\begin{proof}
    Fix $n_0\in \mathbb{N}$ as in Theorem \ref{beyond-bgt} applied to $\Lambda/2$-translation invariant operators. Let $T\in \mathcal{M}_{\Lambda}\left(L^2(\mathbb{R}^d)\right)$, then by  \ref{l2-correspondence}, there exists a $\Lambda/2$-translation invariance operator $S\in \mathcal{L}\left(L^2(\mathbb{R}^d)\right)$ such that $T=SP=\Theta(S).$ We know from Theorem \ref{beyond-bgt} that there exists a net $\{S_i\}_{i\in \mathbb{N}}\in \mathcal{L}\left(L^2(\mathbb{R}^d)\right)$ such that $S_i\to S$ in weak-operator or strong-operator topology. We can proceed similarly to the proof of Corollary \ref{lamb-mod-operator-limit} to show that each $S_iP=\Theta(S_i)$ is a Fourier-Wigner operator-periodisation of a rank-$n_0$ operator generated by functions coming from $\mathcal{E}_{\Lambda/2}(\mathbb{R}^d).$ We are now done since $\Theta$ is continuous with respect to the weak-operator and strong-operator topologies, and we obtain $\Theta(S_i)\to \Theta(S)=T$ in either of these topologies.
\end{proof}
Fix a lattice $\Lambda$ and consider the Heisenberg module $\mathcal{E}_{\Lambda}(\mathbb{R}^d)$, it actually follows from the proof of Lemma \ref{adjointables-are-mix-gab} that the integer $n$ can be fixed to be the minimum number of elements required to generate $\mathcal{E}_{\Lambda}(\mathbb{R}^d)$ as a \emph{finitely generated projective} left $A$-module. We couple this with the fact that the generators of $\mathcal{E}_{\Lambda}(\mathbb{R}^d)$ as a left $A$-module are exactly multi-window Gabor frames for $L^2(\mathbb{R}^d)$ \cite[Theorem 3.16]{AuEn20}, and that the existence of $n$-multi-window Gabor frames on $L^2(\mathbb{R}^d)$ with atoms from $\mathcal{E}_{\Lambda}(\mathbb{R}^d)$ implies that the volume of the lattice $\Lambda$ must necessarily satisfy: $|\Lambda|\leq n$ \cite[Proposition 4.35]{AuJaMaLu20}; we then have the following Lemma.
\begin{lemma}\label{generator-density}
    For a fixed $\Lambda$, if $n\in \mathbb{N}$ is the minimum number of elements required to generate $\mathcal{E}_{\Lambda}(\mathbb{R}^d)$ as a left $A$-module, then $|\Lambda|\leq n.$
\end{lemma}
\noindent We can interpret this result to mean that the sparser the lattice $\Lambda$ is (i.e. larger $|\Lambda|$), the greater number of elements are required to generate $\mathcal{E}_{\Lambda}(\mathbb{R}^d)$. 

Throughout this study, we have seen that there is a fixed integer $n\in \mathbb{N}$ that would allow us to characterise $\Lambda$-translation invariant operators as operator-periodisations of rank-$n$ operators. We have just discussed that this very same $n$ can be fixed to be the minimum number required to generate $\mathcal{E}_{\Lambda}(\mathbb{R}^d)$. Following Lemma \cref{generator-density}, we see that there is a discrepancy between the lowest possible rank of operators that we can use to characterise $\Lambda$-translation invariant operators via operator-periodisation versus that of $\Lambda$-modulation invariant operators via Fourier-Wigner operator-periodisation.
\begin{proposition}\label{generator-discrepancy}
    For a fixed lattice $\Lambda\subseteq \mathbb{R}^d$, we have the following:
    \begin{enumerate}[label=(\alph*)]
        \item If $n$ is the minimum number that makes the characterisation of $\Lambda$-translation invariant operators as in Theorem \ref{lamb-trans-operator-limit} and Theorem \ref{beyond-bgt} true, then $|\Lambda|\leq n.$
        \item If $n_0$ is the minimum number the makes the characterisation of $\Lambda$-modulation invariant operators as in Theorem \ref{lamb-mod-operator-limit} and Theorem \ref{final-mod-vn} true, then $\frac{|\Lambda|}{4^d}\leq n_0.$
    \end{enumerate}
\end{proposition}
\begin{proof}\leavevmode
    \begin{enumerate}[label=(\alph*)]
        \item We know that $n$ is the same as the minimum number of generators required to generate $\mathcal{E}_{\Lambda}(\mathbb{R}^d)$, therefore it follows from Lemma \ref{generator-density} that $|\Lambda|\leq n$.
        \item The proof here is similar to Item 1, however $n_0$ is the minimum number of generators required to generate $\mathcal{E}_{\Lambda/2}(\mathbb{R}^d)$. Therefore it follows from \ref{generator-density} that $\left|\Lambda/2 \right|\leq n_0.$ Since we always consider lattices here to be full--rank lattices, there must exist some $A\in \op{GL}(2d,\mathbb{R})$ such that $\Lambda= A\mathbb{Z}^{2d}$, therefore $|\Lambda/2| = |\det(A/2)|=\left(\frac{1}{2}\right)^{2d}|\det{(A)}|=\left(\frac{1}{4}\right)^d|\Lambda|.$ We then obtain $\frac{|\Lambda|}{4^d}\leq n_0.$
    \end{enumerate}
\end{proof}
\noindent We see from the result above that $\Lambda$-modulation invariant operators generally require less generators from the Heisenberg module to be represented as a Fourier-Wigner periodisation, compared to representations of $\Lambda$-translation invariant operators via operator-periodisation. In fact, in one special case, we have an exact result concerning these generators. Let $\{e_1,...,e_{2d}\}$ be the standard basis for $\mathbb{R}^{2d}$. We say that a lattice $\Lambda=A\mathbb{Z}^{2d}$, $A\in\op{GL}(2d,\mathbb{R})$ is \emph{non-rational} if there exists $i,j\in \{1,...,2d\}$ such that $\Omega(Ae_i,Ae_j)$ is not rational. A curious result concerning non-rational lattices due to \cite[Corollary 5.6]{JaLu21} states that the Heisenberg module $\mathcal{E}_{\Lambda}(\mathbb{R}^d)$ with non-rational lattice $\Lambda$ satisfying $n-1\leq |\Lambda|<n$ for $n\in \mathbb{N}$ has exactly $n$ generators. Therefore, in terms of the integer ceiling function $\ceil{\ \cdot\ }:\mathbb{R}\to \mathbb{Z}$, we have the following result, by an argument similar to Proposition \ref{generator-discrepancy}.
\begin{proposition}
    If $\Lambda$ is non-rational:
    \begin{enumerate}[label=(\alph*)]
        \item Every $\Lambda$-translation invariant operator is characterized as the topological limits of operator-periodisation in the sense of Theorem \ref{lamb-trans-operator-limit} and Theorem \ref{beyond-bgt}, of exactly rank-$\ceil{|\Lambda|}$ operators generated by $\mathcal{E}_{\Lambda}(\mathbb{R}^d).$
        \item Every $\Lambda$-modulation invariant operator is characterized as the topological limits of Fourier-Wigner operator-periodisation in the sense of Theorem \ref{lamb-mod-operator-limit} and Theorem \ref{final-mod-vn}, of exactly rank-$\ceil{|\Lambda/2|}$ operators generated by $\mathcal{E}_{\Lambda/2}(\mathbb{R}^d).$
    \end{enumerate}
\end{proposition}
\section{Discussion}
The interesting appearance of the $\Lambda/2$ lattice in the correspondence between translation and modulation invariant operators seems peculiar at first glance, as one may initially suspect a $\Lambda$ to $\Lambda$ type correspondence. However, the $\Lambda/2$ arises directly as a result of using the Weyl quantisation, as considered in \cref{weylquant}. In fact the need for automorphisms of the type $x\mapsto 2x$ can cause problems for Weyl quantisation on compact or discrete locally compact abelian groups. However, since in this work we consider phase space as the underlying group, we are free to scale lattices as required. Nonetheless, it is instructive to consider how the concept of modulations looks for other quantisation schemes than Weyl. Consider Cohen's class quantisations, that is to say those defined in terms of a convolution with the Wigner distribution:
\begin{align*}
    \langle (a\star L_{\sigma})f, g\rangle = \langle a*\sigma, W(f,g)\rangle.
\end{align*}
Quantisation schemes of this type include $\tau$-Weyl quantisation, itself encompassing Kohn-Nirenberg quantisation and operators “with right symbol”, as well as the Born-Jordan quantisation (cf. \cite{Bo10}). Since these quantisation schemes relate to Weyl quantisation via a convolution, operator translations given by $\alpha_z$ correspond to translations on the symbols of all of these quantisation schemes as a result of the Weyl case;
\begin{align*}
    \langle \pi(z)(a\star L_{\sigma})\pi(z)^*f, g\rangle &= \langle T_{z}(a*\sigma), W(g,f) \rangle \\
    &= \langle a* T_{z}\sigma, W(g,f) \rangle.
\end{align*}
However, considering modulations instead of translation, we clearly lose this universality. For some $\tau$-Weyl quantisation where $\tau\neq 1/2$, 
\begin{align*}
    a_{\tau} = \frac{2^d}{|2\tau-1|}e^{2\pi i\tfrac{2}{2\tau-1}x\cdot\omega}
\end{align*}
so we have the identity
\begin{align*}
    \sigma_S^{\tau} = \frac{2^d}{|2\tau-1|}e^{2\pi i\tfrac{2}{2\tau-1}x\cdot\omega} * \sigma_S
\end{align*}
(cf. \cite{Bo10}). Modulation in the sense of $\tau$-Weyl quantisation must then correspond to a translation of 
\begin{align*}
    \mathcal{F}_{\Omega} (\sigma_S^{\tau}) = \frac{2^d}{|2\tau-1|}\mathcal{F}_{\Omega} (e^{2\pi i\tfrac{2}{2\tau-1}x\cdot\omega}) \cdot \mathcal{F}_{\Omega} (\sigma_S).
\end{align*}
In this sense Weyl quantisation, where the planar wave disappears, can be seen as the natural setting for quantum time--frequency analysis.

\section*{Acknowledgements}
The authors would like to thank Professor Franz Luef for helpful advice.

\bibliographystyle{abbrv}
\bibliography{refs}

\Addresses

\end{document}